\documentclass[12pt]{article}

\usepackage{amsfonts,graphicx,amsmath,amssymb,amsthm,geometry,color,nicefrac,enumerate,relsize,textcomp, hyperref, comment, titling}
\usepackage{mathtools} 
\DeclarePairedDelimiter{\ceil}{\lceil}{\rceil}
\DeclarePairedDelimiter{\floor}{\lfloor}{\rfloor}
\usepackage[latin1]{inputenc}
\usepackage{bbm} 

\newcommand{\R}{\mathbb{R}}

\newcommand{\N}{\mathbb{N}}
\newcommand{\Z}{\mathbb{Z}}
\newcommand{\E}{\mathbb{E}}
\renewcommand{\P}{\mathbb{P}}

\newcommand{\Exp}[1]{\E \!\left[ #1 \right]}
\newcommand{\EXp}[1]{\E [ #1 ]}
\newcommand{\EXP}[1]{\E \big[  #1 \big]}
\newcommand{\EXPP}[1]{\E \Big[  #1 \Big]}

\newcommand{\qand}{\qquad\text{and}}
\newcommand{\qandq}{\qquad\text{and}\qquad}

\newcommand{\norm}[1]{\Vert #1 \Vert}
 
\newcommand{\lrnorm}[1]{\left\| #1 \right\|}

\newcommand{\abs}[1]{| #1 |}  
\newcommand{\Abs}[1]{\big| #1 \big|}
\newcommand{\ABs}[1]{\Big| #1 \Big|}
\renewcommand{\lll}{\langle} 
\newcommand{\rrr}{\rangle}
\newcommand{\llll}{\big\langle}
\newcommand{\rrrr}{\big\rangle}

\renewcommand{\log}{\operatorname{ln}}

\newcommand{\note}[1]

\newtheorem{lemma}{Lemma}[section]

\newtheorem{cor}[lemma]{Corollary}

\newtheorem{theorem}[lemma]{Theorem}

\newtheorem{prop}[lemma]{Proposition}

\newtheorem{setting}[lemma]{Setting}

\usepackage{environ}

\begin{document}
\title{Lower error bounds for the stochastic \\
gradient descent optimization algorithm: \\
Sharp convergence rates for slowly \\
 and fast decaying learning rates}

\author{Arnulf Jentzen$^1$
	and 
	Philippe von Wurstemberger$^{2}$
	\bigskip
	\\
	\small{$^1$Department of Mathematics, 
		ETH Zurich,}
	\\
		\small{e-mail: 
		arnulf.jentzen@sam.math.ethz.ch}
	\smallskip
\\
\small{$^2$Department of Mathematics, 
	ETH Zurich,}
\\
\small{e-mail: 
	vwurstep@student.ethz.ch}	
}

\maketitle

\begin{abstract}
The stochastic gradient descent (SGD) optimization algorithm plays a central role in a series of machine learning applications. 
The scientific literature provides a vast amount of upper error bounds for the SGD method.
Much less attention as been paid to proving lower error bounds for the SGD method.
It is the key contribution of this paper to make a step in this direction.
More precisely, in this article we establish for every $\gamma, \nu \in (0,\infty)$ essentially matching lower and upper bounds for the mean square error of the SGD process with learning rates $(\frac{\gamma}{n^\nu})_{n \in \N}$ associated to a simple quadratic stochastic optimization problem.
This allows us to precisely quantify the mean square convergence rate of the SGD method in dependence on the asymptotic behavior of the learning rates.
\end{abstract}

\newpage

\tableofcontents

\section{Introduction}
\label{sect:intro}
The stochastic gradient descent (SGD) optimization algorithm plays a central role in machine learning and, in particular, deep learning applications such as image analysis and speech recognition 
(cf., e.g., 
\cite{
GravesMohamedHinton13, 
HintonETAL12, 
KrizhevskySutskeverHinton12,
Ruder16}). 
It is therefore important to analyze and quantify the convergence speed of the SGD method. 
There is a vast amount of scientific literature investigating and providing upper bounds for the SGD method and modifications of it
(cf., e.g.,
\cite{
BachMoulines11,
BachMoulines13,
Bottou12,
BottouBousquet11,
BottouCourtisNocedal16,
BottouLeCun04,
ChauKumarRasonyiSabanis17,
DereichGronbach17,
DieuleveutDurmusBach17,
LiTaiE15,
Murata98,
NguyenETAL18,
TangMonteleoni15} and cf., e.g., \cite{mypaper} for a more comprehensive review of the literature). 
Much less attention has been paid to proving lower error bounds for the SGD method, that is, to quantifying the best possible speed of convergence which the SGD method can achieve
(cf., e.g., 
\cite{
AgarwalBottou14,
LanZhou17,
MullerGBRitter08,
RakhlinShamirSridharan12,
WoodworthSrebro16}).
It is the key contribution of this paper to make a step in this direction.

To be more specific, in this paper we precisely quantify the speed of convergence of the SGD process in the case of a simple quadratic stochastic optimization problem (cf.\ item~\eqref{intro_thm:item1} in Theorem~\ref{intro_thm} below) for both slowly as well as fast decaying learning rates. 
In particular, in Theorem~\ref{intro_thm} below we provide for every $\gamma, \nu \in (0,\infty)$ essentially matching upper and lower bounds for the root mean square distance between the global minimum of the considered stochastic optimization problem and the SGD process with the learning rates $(\frac{\gamma}{n^\nu})_{n \in \N}$.
\begin{theorem}
\label{intro_thm}
Let $d \in \N$, $\alpha, \gamma,\nu  \in (0,\infty)$, $\xi \in \R^d$, 
let $\lll \cdot, \cdot \rrr \colon \R^d \times \R^d \to \R$ be the $d$-dimensional Euclidean scalar product, 
let $\lrnorm{\cdot} \colon \R^d \to [0,\infty)$ be the $d$-dimensional Euclidean norm, 
let $(\Omega, \mathcal{F}, \P)$ be a probability space, 
let $X_n \colon \Omega \to \R^d$, $n \in \N$, be i.i.d.\ random variables with
$\Exp{\norm{X_1}^2} < \infty$
and
$\P(X_1= \EXp{X_1}) < 1$,
let $(r_{\varepsilon, i})_{\varepsilon \in (0,\infty), i \in \{0,1\}} \subseteq \R$ satisfy for all $\varepsilon \in (0,\infty)$, $i \in \{0,1\}$ that
\begin{equation}
r_{\varepsilon, i} =
\begin{cases}
	\nicefrac{\nu}{2} 																			&\colon \nu < 1 \\
	\min \{\nicefrac{1}{2}, \gamma \alpha + (-1)^{i} \varepsilon\}	&\colon \nu = 1  \\
	0																										&\colon \nu > 1,
\end{cases}
\end{equation}
let $F = ( F(\theta,x) )_{(\theta,x) \in \R^d \times \R^d} \colon \R^d \times \R^d \to \R$ and $f \colon \R^d \to \R$ be the functions which satisfy for all $\theta, x  \in \R^d$ that
\begin{equation}
\label{intro_thm:ass1}
F(\theta,x) = \tfrac{\alpha}{2} \norm{\theta-x}^2 \qandq f(\theta) = \EXP{F(\theta,X_1)},
\end{equation}
and let $\Theta \colon \N_0 \times \Omega \to \R^d$ be the stochastic process which satisfies for all $n \in \N$ that 
\begin{equation}
\label{intro_thm:ass2}
\begin{split}
\Theta_0 = \xi \qandq \Theta_n = \Theta_{n-1} - \tfrac{\gamma}{n^\nu} (\nabla_\theta F) (\Theta_{n-1},X_n).
\end{split}
\end{equation}
Then 
\begin{enumerate}[(i)]
\item \label{intro_thm:item1}
there exists a unique $\vartheta \in \R^d$ such that 
$
\{\theta \in \R^d  \colon  f(\theta) = \inf\nolimits_{w \in \R^d} f(w)  \} = \{ \vartheta \},
$

\item \label{intro_thm:item2}
for every $\varepsilon \in (0,\infty)$ there exist $c_0,c_1 \in (0,\infty)$ such that for all $n \in \N$ it holds that
\begin{equation}
c_0n^{-r_{\varepsilon,0}}
\leq
\big(\EXP{\norm{\Theta_n-\vartheta}^2}\big)^{\nicefrac{1}{2}} 
\leq
c_1  n^{-r_{\varepsilon,1}},
\end{equation}

and
\item \label{intro_thm:item3}
for every $\varepsilon \in (0,\infty)$ there exist $C_0,C_1 \in (0,\infty)$ such that for all $n \in \N$ it holds that
\begin{equation}
C_0  n^{-2r_{\varepsilon,0}} 
\leq
\Exp{f(\Theta_n)} - f(\vartheta) 
\leq
C_1 n^{-2r_{\varepsilon,1}}.
\end{equation}
\end{enumerate}
\end{theorem}

Theorem~\ref{intro_thm} is an immediate consequence of Theorem~\ref{main_theorem} below, which is the main result of this article.
We now roughly describe the dependence, exhibited in Theorem~\ref{intro_thm}, of the root mean square convergence rate of the SGD process on the learning rates.
In the case of slowly decaying  learning rates 
(corresponding to the case $\nu< 1$ in Theorem~\ref{intro_thm}) 
the convergence rate of the SGD process does not depend on the size of the learning rates 
(corresponding to the parameter $\gamma$ in Theorem~\ref{intro_thm}). 
In this case we note that faster decay of the learning rates (corresponding to larger $\nu$ in Theorem~\ref{intro_thm}) results in a higher convergence rate of the SGD process.
In the case of fast decaying but large learning rates 
(corresponding to the case $\nu = 1$ and $\gamma > \nicefrac{1}{(2\alpha)}$ in Theorem~\ref{intro_thm})
 the SGD process attains the optimal convergence rate of $\nicefrac{1}{2}$.
In the case of fast decaying and small learning rates
(corresponding to the case $\nu = 1$ and $\gamma \leq \nicefrac{1}{(2\alpha)}$ in Theorem~\ref{intro_thm})
the convergence rate of the SGD process depends on the size of the learning rates. 
In this case we observe that the smaller the learning rates are
(corresponding to smaller $\gamma$ in Theorem~\ref{intro_thm})
the lower is the resulting convergence rate.
Note that this is contrary to the effect observed above in the case of slowly decaying learning rates.
The phenomenon that the convergence rate increases as $\nu$ increases (so that the learning rates get smaller) but also increases as $\gamma$ increases (so that the learning rates get larger), roughly speaking, arises from the interplay of two sources of errors: The error due to the randomness in the SGD method (which gets smaller when the learning rates get smaller) and the error due to the fact that the deterministic gradient method does not reach in finite time the whole infinite time interval of the underlying gradient flow (which gets smaller when the learning rates get larger).
Finally, in the case of very fast decaying learning rates (corresponding to the case $\nu > 1$ in Theorem~\ref{intro_thm}) the SGD process fails to converge to the global minimum of the objective function.


The remainder of this paper is organized as follows.
In Section~\ref{sect:basic} we introduce the setting of the stochastic optimization problem considered in this paper and we establish a few basic properties for the objective function, the loss function, and the SGD process.
In Section~\ref{sect:UB} we derive upper bounds for the root mean square error of the SGD process.
In Section~\ref{sect:LB} we first establish in Subsections~\ref{subsect:LB_slow_and_fast}--\ref{subsect:LB_very_fast} lower bounds for the root mean square error of the SGD process which essentially match the upper bounds of Section~\ref{sect:UB}. 
Then, in Subsection~\ref{subsect:main_result}, we combine the upper and lower bounds of this article in Theorem~\ref{main_theorem} and thereby obtain a sharp convergence rate of the SGD process in dependence of the learning rates. 
Theorem~\ref{intro_thm} above is an immediate consequence of Theorem~\ref{main_theorem}.

\section[Basic properties for stochastic gradient descent (SGD)]{Basic properties for the stochastic gradient descent (SGD) optimization method}
\label{sect:basic}

In Section~\ref{sect:UB} and Section~\ref{sect:LB} below we provide a detailed error analysis for the SGD optimization method in the case of a simple quadratic loss function; cf., particularly, Theorem~\ref{main_theorem} below. 
In this section we introduce the setting of the considered optimization problem (see Setting~\ref{setting} in Subsection~\ref{subsect:setting} below) and we establish some elementary properties for the optimization problem under consideration (see Lemma~\ref{properties_loss} below) and the associated SGD process (see Proposition~\ref{explicit_expressions} below). These elementary properties will be repeatedly used in the convergence rate proofs in our detailed error analysis in Section~\ref{sect:UB} and Section~\ref{sect:LB} below.

\subsection{Setting}
\label{subsect:setting}
Throughout this article the following setting is frequently used.

\begin{setting}
\label{setting}
Let $d \in \N$, $\alpha, \gamma, \nu\in (0,\infty)$, $\xi \in \R^d$, 
let $\lll \cdot, \cdot \rrr \colon \R^d \times \R^d \to \R$ be the $d$-dimensional Euclidean scalar product, 
let $\lrnorm{\cdot} \colon \R^d \to [0,\infty)$ be the $d$-dimensional Euclidean norm, 
let $(\Omega, \mathcal{F}, \P)$ be a probability space, 
let $X_n \colon \Omega \to \R^d$, $n \in \N$, be i.i.d.\ random variables with $\Exp{\norm{X_1}^2} < \infty$, 
let $F = ( F(\theta,x) )_{(\theta,x) \in \R^d \times \R^d} \colon \R^d \times \R^d \to \R$ and $f \colon \R^d \to \R$ be the functions which satisfy for all $\theta, x  \in \R^d$ that
\begin{equation}
\label{setting:eq1}
F(\theta,x) = \tfrac{\alpha}{2} \norm{\theta-x}^2 \qandq f(\theta) = \EXP{F(\theta,X_1)},
\end{equation}
and let $\Theta \colon \N_0 \times \Omega \to \R^d$ be the stochastic process which satisfies for all $n \in \N$ that 
\begin{equation}
\label{setting:eq2}
\begin{split}
\Theta_0 = \xi \qandq \Theta_n = \Theta_{n-1} - \tfrac{\gamma}{n^\nu} (\nabla_\theta F) (\Theta_{n-1},X_n).
\end{split}
\end{equation}
\end{setting}

\subsection{Basic properties of the objective and the loss function}
\label{subsect:basic}

In this subsection we establish in Lemma~\ref{properties_loss} below some basic properties for the objective and the loss function of the optimization problem under consideration (cf.\ Setting~\ref{setting} above). Our proof of Lemma~\ref{properties_loss} employs the elementary and well-known results in Lemma~\ref{L2_distance} and Lemma~\ref{der_of_norm}. For completeness we also provide the proofs of Lemma~\ref{L2_distance} and Lemma~\ref{der_of_norm} here. 

\subsubsection{Bias-variance decomposition of the mean square error}
\label{subsubsect:L2_distance}

\begin{lemma}
\label{L2_distance}
Let $d \in \N$, $\vartheta \in \R^d$,
let $\lll \cdot, \cdot \rrr \colon \R^d \times \R^d \to \R$ be a scalar product,
let $\lrnorm{\cdot} \colon \R^d \to [0,\infty)$ be the function which satisfies for all $v \in \R^d$ that $\lrnorm{v} = \sqrt{\lll v, v \rrr}$,
let $(\Omega, \mathcal{F}, \P)$ be a probability space, and
let $Z \colon \Omega \to \R^d$ be a random variable with $\EXp{\norm{Z}} < \infty$.
Then 
\begin{equation}
\Exp{\norm{Z-\vartheta}^2} = \Exp{\norm{Z-\EXp{Z}}^2} + \norm{\EXp{Z} - \vartheta}^2.
\end{equation}
\end{lemma}

\begin{proof}[Proof of Lemma~\ref{L2_distance}]
Observe that the hypothesis that $\EXp{\norm{Z}} < \infty$ and the Cauchy-Schwarz inequality ensure that 
\begin{equation}
\begin{split}
\EXP{ \abs{\lll Z - \EXp{Z}, \EXp{Z}-\vartheta\rrr} } 
&\leq 
\EXP{\norm{ Z - \EXp{Z}}\norm{\EXp{Z}-\vartheta } }\\ 
&\leq
 (\Exp{\norm{ Z}} + \norm{\EXp{Z}})\norm{\EXp{Z}-\vartheta } < \infty.
\end{split}
\end{equation}
The linearity of the expectation hence shows that
\begin{equation}
\begin{split}
&\EXP{\norm{Z-\vartheta}^2} 
=
\EXP{\norm{(Z - \EXp{Z}) + (\EXp{Z}-\vartheta)}^2} \\
&=
\EXP{\norm{Z - \EXp{Z}}^2 + 2\lll Z - \EXp{Z}, \EXp{Z}-\vartheta\rrr + \norm{\EXp{Z}-\vartheta}^2} \\
&=
\EXP{\norm{Z - \EXp{Z}}^2} + 2\lll \EXp{Z} - \EXp{Z}, \EXp{Z}-\vartheta\rrr + \norm{\EXp{Z}-\vartheta}^2 \\
&=
\EXP{\norm{Z - \EXp{Z}}^2} + \norm{\EXp{Z}-\vartheta}^2.
\end{split}
\end{equation}
The proof of Lemma~\ref{L2_distance} is thus completed.
\end{proof}

\subsubsection{On the derivative of the Euclidean norm}
\label{subsubsect:der_of_norm}

\begin{lemma}[Derivative of the Euclidean norm]
\label{der_of_norm}
Let $d \in \N$, $\vartheta \in \R^d$,
let $\lrnorm{\cdot} \colon \R^d \to [0,\infty)$ be the $d$-dimensional Euclidean norm, 
and let $f \colon \R^d \to \R$ be the function which satisfies for all $\theta \in \R^d$ that
\begin{equation}
f(\theta) = \norm{\theta- \vartheta}^2.
\end{equation}
Then it holds for all $\theta \in \R^d$ that $f \in C^{\infty}(\R^d, \R)$  and
\begin{equation}
(\nabla f)(\theta) = 2 (\theta - \vartheta).
\end{equation}
\end{lemma}

\begin{proof}[Proof of Lemma~\ref{der_of_norm}]
Throughout this proof let $\vartheta_1,\ldots, \vartheta_d \in \R$ satisfy that $\vartheta = (\vartheta_1,\ldots, \vartheta_d)$.
Note that the fact that for all $\theta = (\theta_1,\ldots,\theta_d) \in \R^d$ it holds that
\begin{equation}
f(\theta) = \sum_{i = 1}^d \big[\theta_i - \vartheta_i \big]^2
\end{equation}
implies that for all $\theta = (\theta_1,\ldots,\theta_d) \in \R^d$ it holds that $f \in C^{\infty}(\R^d, \R)$ and
\begin{equation}
(\nabla f)(\theta) =
\begin{pmatrix}
\big(\tfrac{\partial f}{ \partial \theta_1}\big)(\theta) \\
\vdots \\
\big(\tfrac{\partial f}{ \partial \theta_d}\big)(\theta)
\end{pmatrix} 
=
\begin{pmatrix}
2(\theta_1- \vartheta_1)\\
\vdots \\
2(\theta_d- \vartheta_d)
\end{pmatrix} 
=
2(\theta-\vartheta).
\end{equation}
The proof of Lemma~\ref{der_of_norm} is thus completed.
\end{proof}

\subsubsection{Basic properties of the objective and the loss function}
\label{subsubsect:properties_loss}

\begin{lemma}
\label{properties_loss}
Assume Setting~\ref{setting}.
Then
\begin{enumerate}[(i)]
\item \label{properties_loss:item1}
it holds for all $\theta \in \R^d$ that
$
f(\theta)
= 
\tfrac{\alpha}{2} \norm{\theta- \EXp{X_1}}^2  + 
\tfrac{\alpha}{2} \, \EXP{\norm{X_1 - \EXp{X_1}}^2},
$

\item \label{properties_loss:item2}
it holds that
$
\{\theta \in \R^d  \colon  f(\theta) = \inf\nolimits_{w \in \R^d} f(w)  \} = \{ \EXp{X_1} \},
$

\item \label{properties_loss:item3}
it holds for all $\theta, x \in \R^d$ that
$
(\nabla_\theta F)(\theta,x) = \alpha(\theta-x) ,
$

\item \label{properties_loss:item4}
it holds for all $\theta \in \R^d$ that
$
(\nabla f)(\theta) =  \Exp{(\nabla_{\theta}F)(\theta, X_1)} = \alpha(\theta - \Exp{X_1}),
$

\item \label{properties_loss:item5}
it holds for all $\theta \in \R^d$ that
$
\lll \theta - \EXp{X_1},(\nabla f)(\theta) \rrr   =  \alpha \norm{\theta - \EXp{X_1}}^2,
$

\item \label{properties_loss:item6}
it holds for all $\theta \in \R^d$ that
$
\norm{(\nabla f)(\theta)} =    \alpha \norm{\theta - \EXp{X_1}},
$
and

\item \label{properties_loss:item7}
it holds for all $\theta \in \R^d$ that
\begin{equation}
\EXP{\norm{(\nabla_\theta F) (\theta,X_1) - (\nabla f)(\theta)}^2} = \alpha^2 \, \EXP{ \norm{X_1 - \EXp{X_1}}^2}.
\end{equation}
\end{enumerate}
\end{lemma}

\begin{proof}[Proof of Lemma~\ref{properties_loss}]
First, note that the hypothesis that $\Exp{\norm{X_1}^2} < \infty$ and Lemma~\ref{L2_distance} (with $\vartheta = \theta$, $Z = X_1$ in the notation of Lemma~\ref{L2_distance}) ensure that for all $\theta \in \R^d$ it holds that
\begin{equation}
\begin{split}
f(\theta)
&= 
\EXP{F(\theta,X_1)}
=
\tfrac{\alpha}{2} \,\EXP{ \norm{ X_1 - \theta}^2} \\
&=  
\tfrac{\alpha}{2}\left( \EXP{\norm{X_1 - \EXp{X_1}}}^2 + \norm{\theta- \EXp{X_1}}^2 \right).
\end{split}
\end{equation}
This establishes item~\eqref{properties_loss:item1}.
Next observe that item~\eqref{properties_loss:item1} proves item~\eqref{properties_loss:item2}.
In addition, note that Lemma~\ref{der_of_norm} proves that for all $\theta ,x\in \R^d$ it holds that
\begin{equation}
(\nabla_\theta F)(\theta,x) 
= 
\tfrac{\alpha}{2} ( 2 (\theta- x))
=
\alpha(\theta-x).
\end{equation}
This establishes item~\eqref{properties_loss:item3}.
Moreover, observe that Lemma~\ref{der_of_norm}, item~\eqref{properties_loss:item1}, and item~\eqref{properties_loss:item3} ensure that for all $\theta \in \R^d$ it holds that
\begin{equation}
\begin{split}
(\nabla f)(\theta) 
&= 
\tfrac{\alpha}{2}( 2(\theta - \EXp{X_1}) )
=
\alpha(\theta - \Exp{X_1}) \\
&=
\Exp{\alpha(\theta - X_1)}
=
\Exp{(\nabla_{\theta}F)(\theta, X_1)}.
\end{split}
\end{equation}
This proves item~\eqref{properties_loss:item4}.
Next note that item~\eqref{properties_loss:item4} implies items~\eqref{properties_loss:item5}--\eqref{properties_loss:item6}.
Moreover, note that  item~\eqref{properties_loss:item3} and item~\eqref{properties_loss:item4} demonstrate that for all $\theta \in \R^d$ it holds that
\begin{equation}
\begin{split}
\EXP{\norm{(\nabla_\theta F) (\theta,X_1) - (\nabla f)(\theta)}^2} 
&=
\EXP{\norm{\alpha(\theta - X_1) - \alpha(\theta - \EXp{X_1})}^2} \\
&=
\alpha^2 \,\EXP{ \norm{X_1 - \EXp{X_1}}^2}.
\end{split}
\end{equation}
This establishes item~\eqref{properties_loss:item7}.
The proof of Lemma~\ref{properties_loss} it thus completed.
\end{proof}

\subsection{On explicit formulas for the SGD process}
\label{subsect:explicit}

In this subsection we establish in Proposition~\ref{explicit_expressions} below a few explicit formulas for the SGD process in \eqref{setting:eq2}. 
Our proof of Proposition~\ref{explicit_expressions} employs the elementary and well-known result for affine recursions in Lemma~\ref{recursive_eq} below. 
For completeness we also present the proof of Lemma~\ref{recursive_eq} here.

\subsubsection{On a recursive equality}
\label{subsubsect:recursive_eq}

\begin{lemma}
\label{recursive_eq}
Let $ d\in \N$, 
$ (\alpha_{n})_{n \in \N}\subseteq \R$, $(\beta_{n})_{n \in \N} \subseteq \R^d$, $(e_n)_{n \in \N_0} \subseteq \R^d $ satisfy for all $n \in \N$ that
\begin{equation}
\label{recursive_eq:ass1}
e_{n} = \alpha_{n} e_{n-1} + \beta_{n} .
\end{equation}
Then it holds for all $n \in \N_0$ that
\begin{equation}
\label{recursive_eq:concl1}
e_{n} = \left[\prod_{l = 1}^{n}\alpha_l \right] e_0 + \sum_{k = 1}^{n}\left(\left[\prod_{l = k+1}^{n}\alpha_l \right] \beta_k\right).
\end{equation}
\end{lemma}

\begin{proof}[Proof of Lemma~\ref{recursive_eq}]
We prove (\ref{recursive_eq:concl1}) by induction on $n \in \N_0$. For the base case $n = 0$ observe that
\begin{equation}
\left[\prod_{l = 1}^{0}\alpha_l \right] e_0 + \sum_{k = 1}^{0}\left(\left[\prod_{l = k+1}^{0}\alpha_l \right] \beta_k\right) = 1 \cdot e_0 + 0 = e_0.
\end{equation}
This establishes (\ref{recursive_eq:concl1}) in the case case $n = 0$. 
For the induction step $\N_0 \ni (n-1) \to n \in \N$ note that (\ref{recursive_eq:ass1}) implies that for all $n \in \N$ with 
$
e_{n-1} = \left[\prod_{l = 1}^{n-1}\alpha_l \right] e_0 + \sum_{k = 1}^{n-1}\left(\left[\prod_{l = k+1}^{n-1}\alpha_l \right] \beta_k\right)
$
it holds that
\begin{equation}
\begin{split}
e_{n} 
&= 
\alpha_{n} e_{n-1} + \beta_{n} \\
&=
\alpha_{n} \left(\left[\prod_{l = 1}^{n-1}\alpha_l \right] e_0 + \sum_{k = 1}^{n-1}\left(\left[\prod_{l = k+1}^{n-1}\alpha_l \right] \beta_k\right)\right) + \beta_{n} \\
&=
\left[\prod_{l = 1}^{n}\alpha_l \right] e_0 + \sum_{k = 1}^{n-1}\left(\left[\prod_{l = k+1}^{n}\alpha_l \right] \beta_k\right) + \left[\prod_{l = n+1}^{n} \alpha_l \right] \beta_{n} \\
&=
\left[\prod_{l = 1}^{n}\alpha_l \right] e_0 + \sum_{k = 1}^{n}\left(\left[\prod_{l = k+1}^{n}\alpha_l \right] \beta_k\right).
\end{split}
\end{equation}
Induction thus establishes (\ref{recursive_eq:concl1}).
The proof of Lemma~\ref{recursive_eq} is thus completed.
\end{proof}

\subsubsection{Explicit formulas for the SGD process}
\label{subsubsect:explicit}

\begin{prop}
\label{explicit_expressions}
Assume Setting~\ref{setting}.
Then
\begin{enumerate}[(i)]
\item \label{explicit_expressions:item1}
it holds for all $n \in \N$ that
$
\Theta_n= (1 - \tfrac{\gamma\alpha}{n^\nu})\Theta_{n-1} + \tfrac{\gamma\alpha}{n^\nu} X_n,
$

\item \label{explicit_expressions:item2}
it holds for all $n \in \N_0$ that
\begin{equation}
\label{explicit_expressions:conclusion1}
\Theta_n = \left[ \prod_{l = 1}^n (1-\tfrac{\gamma\alpha}{l^\nu}) \right] \xi + \sum_{k = 1}^n 
\left(
\tfrac{\gamma\alpha}{k^\nu}\left[ \prod_{l = k+1}^n (1-\tfrac{\gamma\alpha}{l^\nu})\right] X_k
\right),
\end{equation}

\item\label{explicit_expressions:item3}
it holds for all $n \in \N$ that
\begin{equation}
\begin{split}
&\Exp{\norm{\Theta_n - \EXp{X_1}}^2} \\
&= 
(1-\tfrac{\gamma\alpha}{n^\nu})^2 \, \Exp{\norm{\Theta_{n-1} - \EXp{X_1}}^2} + ( \tfrac{\gamma\alpha}{n^\nu})^2 \, \Exp{\norm{X_1 - \EXp{X_1}}^2} ,
\end{split}
\end{equation}

and
\item\label{explicit_expressions:item4}
it holds for all $n \in \N_0$ that
\begin{multline}
\infty > \; \Exp{\norm{\Theta_n - \EXp{X_1}}^2} 
=
\left[ \prod_{l = 1}^n (1-\tfrac{\gamma\alpha}{l^\nu})\right]^2 \! \norm{\xi-\EXp{X_1}}^2 \\
 +\Exp{\norm{X_1 - \EXp{X_1}}^2}
\left[
\sum_{k = 1}^n 
\left[
\tfrac{\gamma\alpha}{k^\nu}\left( \prod_{l = k+1}^n (1-\tfrac{\gamma\alpha}{l^\nu})\right)
\right]^2
\right].
\end{multline}

\end{enumerate}
\end{prop}

\begin{proof}[Proof of Proposition~\ref{explicit_expressions}]
First of all, observe that item~\eqref{properties_loss:item3} in Lemma~\ref{properties_loss} assures that for all $n \in \N$ it holds that
\begin{equation}
\begin{split}
\Theta_n &= \Theta_{n-1} - \tfrac{\gamma}{n^\nu} (\nabla_\theta F) (\Theta_{n-1},X_n) \\
&= \Theta_{n-1} - \tfrac{\gamma \alpha}{n^\nu}(\Theta_{n-1}- X_n) \\
&= (1- \tfrac{\gamma \alpha}{n^\nu}) \Theta_{n-1} + \tfrac{\gamma \alpha}{n^\nu}X_n.
\end{split}
\end{equation}
This establishes item~\eqref{explicit_expressions:item1}.
Next note that Lemma~\ref{recursive_eq} (with 
$d = d$,
$(\alpha_n)_{n \in \N} = (1- \tfrac{\gamma \alpha}{n^\nu})_{n \in \N} $,
$(\beta_n)_{n \in \N} = (\tfrac{\gamma \alpha}{n^\nu}X_n(\omega))_{n \in \N}$,
$(e_n)_{n \in \N_0} = (\Theta_n(\omega))_{n \in \N_0}$
for $\omega \in \Omega$ in the notation of Lemma~\ref{recursive_eq}) 
and item~\eqref{explicit_expressions:item1} demonstrate that for all $n \in \N_0$, $\omega \in \Omega$ it holds that
\begin{equation}
\begin{split}
\Theta_n(\omega) 
&=
\left[ \prod_{l = 1}^n (1-\tfrac{\gamma\alpha}{l^\nu}) \right] \Theta_0(\omega) + \sum_{k = 1}^n 
\left(
\left[ \prod_{l = k+1}^n (1-\tfrac{\gamma\alpha}{l^\nu})\right] \left(\tfrac{\gamma\alpha}{k^\nu} X_k(\omega) \right)
\right) \\
&= 
\left[ \prod_{l = 1}^n (1-\tfrac{\gamma\alpha}{l^\nu}) \right] \xi + \sum_{k = 1}^n 
\left(
\tfrac{\gamma\alpha}{k^\nu}\left[ \prod_{l = k+1}^n (1-\tfrac{\gamma\alpha}{l^\nu})\right] X_k(\omega)
\right).
\end{split}
\end{equation}
This proves item~\eqref{explicit_expressions:item2}.  
Furthermore, note that item~\eqref{explicit_expressions:item2} and the fact that 
$ \forall \, k \in \N \colon \EXp{\norm{X_k}^2} = \EXp{\norm{X_1}^2} < \infty$
assure that for all $n \in \N_0$ it holds that $\EXp{\norm{\Theta_n}^2} < \infty$.
This ensures that for all $n \in \N_0$ it holds that
\begin{equation}
\label{explicit_expressions:eq1}
\EXp{\norm{\Theta_n - \EXp{X_1}}^2} < \infty.
\end{equation}
The fact that 
$ \forall \, k \in \N \colon \EXp{\norm{X_k}^2} = \EXp{\norm{X_1}^2} < \infty$
 and item~\eqref{explicit_expressions:item1} therefore imply that for all $n \in \N$ it holds that
\begin{equation}
\label{explicit_expressions:eq2}
\begin{split}
&\Exp{\norm{\Theta_n-\EXp{X_1}}^2} \\
&= 
\Exp{\norm{(1- \tfrac{\gamma \alpha}{n^\nu}) \Theta_{n-1} + \tfrac{\gamma \alpha}{n^\nu}X_n-\EXp{X_1}}^2} \\
&=
\Exp{\norm{(1- \tfrac{\gamma \alpha}{n^\nu}) (\Theta_{n-1} - \EXp{X_1})  + \tfrac{\gamma \alpha}{n^\nu}(X_n-\EXp{X_1})}^2} \\
&=
\EXPP{\norm{(1- \tfrac{\gamma \alpha}{n^\nu}) (\Theta_{n-1} - \EXp{X_1})}^2  \\
&\quad +
2 \, \lll (1- \tfrac{\gamma \alpha}{n^\nu}) (\Theta_{n-1} - \EXp{X_1}), \tfrac{\gamma \alpha}{n^\nu}(X_n-\EXp{X_1})\rrr + 
 \norm{\tfrac{\gamma \alpha}{n^\nu}(X_n-\EXp{X_1})}^2} \\
&=
 (1- \tfrac{\gamma \alpha}{n^\nu})^2 \,\EXP{\norm{ \Theta_{n-1} - \EXp{X_1}}^2}  +  (\tfrac{\gamma \alpha}{n^\nu} )^2 \,\EXP{\norm{X_n-\EXp{X_1}}^2}\\
&\quad +
2 (1- \tfrac{\gamma \alpha}{n^\nu}) (\tfrac{\gamma \alpha}{n^\nu} ) \EXP{\lll \Theta_{n-1} - \EXp{X_1}, X_n-\EXp{X_1}\rrr}.
\end{split}
\end{equation}
In addition, note that the fact that for all independent random variables $Y,Z \colon \Omega \to \R$ with $\EXp{ \abs{Y} + \abs{Z}} < \infty$ it holds that 
$\EXp{\abs{YZ}} < \infty$ and $\EXp{YZ} = \EXp{Y} \, \EXp{Z}$ (cf., e.g., Klenke \cite[Theorem 5.4]{Klenke14}), the fact that for all $n \in \N$ it holds that $\Theta_{n-1}$ and $X_n$ are independent, and the fact that for all $n \in \N$ it holds that $\Exp{ \norm{\Theta_{n-1}} + \norm{X_n}} < \infty$ assure that for all $n \in \N$ it holds that
\begin{equation}
\begin{split}
\EXP{\lll \Theta_{n-1} - \EXp{X_1}, X_n-\EXp{X_1}\rrr} 
&= 
\big< \EXP{\Theta_{n-1} - \EXp{X_1}}, \EXP{X_1-\EXp{X_1}} \!\big\rangle \\
&=
\llll \EXp{\Theta_{n-1}} - \EXp{X_1}, \EXp{X_1}-\EXp{X_1}   \rrrr 
= 0.
\end{split}
\end{equation}
This, \eqref{explicit_expressions:eq2}, and the fact that $(X_n)_{n \in \N}$ are i.i.d\ random variables demonstrate that for all $n \in \N$ it holds that
\begin{equation}
\begin{split}
&\Exp{\norm{\Theta_n-\EXp{X_1}}^2}\\
&=
(1- \tfrac{\gamma \alpha}{n^\nu})^2 \,\EXP{\norm{\Theta_{n-1} - \EXp{X_1}}^2 } 
+ (\tfrac{\gamma \alpha}{n^\nu} )^2 \,\Exp{\norm{X_n-\EXp{X_1}}^2} \\
&=
(1- \tfrac{\gamma \alpha}{n^\nu})^2 \,\EXP{\norm{\Theta_{n-1} - \EXp{X_1}}^2 } 
+ (\tfrac{\gamma \alpha}{n^\nu} )^2 \,\Exp{\norm{X_1-\EXp{X_1}}^2}.
\end{split}
\end{equation}
This proves item~\eqref{explicit_expressions:item3}.
Combining Lemma~\ref{recursive_eq} (with 
$d = 1$, 
$(\alpha_n)_{n \in \N} = ((1- \tfrac{\gamma \alpha}{n^\nu})^2)_{n \in \N}$, 
$(\beta_n)_{n \in \N} =  ((\tfrac{\gamma \alpha}{n^\nu} )^2 \,\Exp{\norm{X_1-\EXp{X_1}}^2})_{n \in \N}$, 
$(e_n)_{n \in \N_0} = (\Exp{\norm{\Theta_n-\EXp{X_1}}^2})_{n \in \N_0}$ 
in the notation of Lemma \ref{recursive_eq}) 
with item~\eqref{explicit_expressions:item3} and \eqref{explicit_expressions:eq1} demonstrates that for all $n \in \N_0$ it holds that
\begin{equation}
\begin{split}
\infty &>\, \Exp{\norm{\Theta_n-\EXp{X_1}}^2} \\
&= 
\left[\prod_{l = 1}^{n}(1- \tfrac{\gamma \alpha}{l^\nu})^2 \right]\Exp{\norm{\Theta_0-\EXp{X_1}}^2} \\
&\quad + \sum_{k = 1}^{n}\left[\left(\prod_{l = k+1}^{n}(1- \tfrac{\gamma \alpha}{l^\nu})^2\right) (\tfrac{\gamma \alpha}{k^\nu} )^2 \,\Exp{\norm{X_1-\EXp{X_1}}^2}\right]\\
&=
\left[ \prod_{l = 1}^n (1-\tfrac{\gamma\alpha}{l^\nu})\right]^2 \! \norm{\xi-\EXp{X_1}}^2 \\
&\quad + 
\Exp{\norm{X_1 - \EXp{X_1}}^2}
\left[
\sum_{k = 1}^n 
\left[
\tfrac{\gamma\alpha}{k^\nu}\left( \prod_{l = k+1}^n (1-\tfrac{\gamma\alpha}{l^\nu})\right)
\right]^2
\right].
\end{split}
\end{equation}
This establishes item~\eqref{explicit_expressions:item4}.
The proof of Proposition~\ref{explicit_expressions} it thus completed.
\end{proof}


\section{Upper error estimates for the SGD optimization method}
\label{sect:UB}
In this section we establish in Proposition~\ref{UB_nuleq1} and Corollary~\ref{UB_nueq1_3} below upper bounds for the root mean square distance between the SGD process in \eqref{setting:eq2} and the global minimum of the considered optimization problem (cf.\ item~\eqref{properties_loss:item2} in Lemma~\ref{properties_loss}).
In our analysis we distinguish between the case of slowly decaying learning rates (see Subsection~\ref{subsect:UB_slow} below), the case of fast decaying learning rates (see Subsection~\ref{subsect:UB_fast} and Subsection~\ref{subsect:UB_fast_refined} below), and the case of very fast decaying learning rates (see Subsection~\ref{subsect:UB_very_fast}  below).

\subsection{Upper errors estimates in the case of slowly decaying learning rates}
\label{subsect:UB_slow}
In this subsection we establish in Proposition~\ref{UB_nuleq1} below an upper bound for the root mean square error of the SGD process in \eqref{setting:eq2} in the case of slowly decaying learning rates (corresponding to the case $\nu< 1$ in Setting~\ref{setting}).
In our proof of Proposition~\ref{UB_nuleq1} we employ the auxiliary and elementary results in Lemma~\ref{recursive_ineq} and Lemma~\ref{aux_UB_nuleq1} below.
A result similar to Lemma~\ref{recursive_ineq} can, e.g., be found in \cite[Corollary 2.18]{mypaper} and a result similar to Lemma~\ref{aux_UB_nuleq1} can, e.g., be found in \cite[Lemma 4.1]{mypaper}.

\subsubsection{On a recursive inequality and an a priori estimate}
\label{subsubsect:recursive_ineq}


\begin{lemma}
\label{recursive_ineq}
Let $\kappa \in [0,\infty)$, $(e_n)_{n \in \N_0} \subseteq [0,\infty)$, $(\gamma_n)_{n \in \N} \subseteq (0,\infty)$ satisfy for all $n \in \N$ that
\begin{equation}
\label{recursive_ineq:assumption1}
e_n \leq (1-\gamma_n)^2 e_{n-1} +  \kappa (\gamma_n)^2\qquad \qand
\end{equation}
\begin{equation}
\label{recursive_ineq:assumption2}
 \liminf_{l \to \infty}\left[\frac{\gamma_l - \gamma_{l-1}}{(\gamma_l)^{2}} + \frac{2\gamma_{l-1}}{\gamma_{l}} - \gamma_{l-1} \right] > 0.
\end{equation}
Then there exists $C \in (0,\infty)$ such that for all $n \in \N$ it holds that
\begin{equation}
\label{recursive_ineq:conclusion}
e_n \leq C\gamma_n.
\end{equation}
\end{lemma}

\begin{proof}[Proof of Lemma~\ref{recursive_ineq}]
Throughout this proof let $m \in \N\cap (1,\infty)$, $\mathcal{I} \in (0,\infty)$ satisfy that
\begin{equation}
\label{recursive_ineq:eq1}
\mathcal{I} =  \inf_{l \in \N \cap (m,\infty)}\left[\frac{\gamma_l - \gamma_{l-1}}{(\gamma_l)^{2}} + \frac{2\gamma_{l-1}}{\gamma_{l}} - \gamma_{l-1} \right] > 0
\end{equation}
(cf.\ \eqref{recursive_ineq:assumption2}) and let $C \in [0,\infty)$ be given by 
\begin{equation}
\label{recursive_ineq:eq2}
C = \max \! \left\{ \frac{e_m}{\gamma_m}, \frac{\kappa}{\mathcal{I}}  \right\}.
\end{equation}
We claim that for all $n \in \{m,m+1,\ldots \}$ it holds that
\begin{equation}
\label{recursive_ineq:eq3}
e_n \leq C \gamma_n.
\end{equation}
We now prove \eqref{recursive_ineq:eq3} by induction on $n \in \{m,m+1,\ldots \}$.
For the base case $n = m$ note that
\begin{equation}
e_m = \left[\frac{e_m}{\gamma_m} \right] \gamma_m \leq C \gamma_m.
\end{equation}
This establishes \eqref{recursive_ineq:eq3} in the base case $n = m$.
For the induction step $ \{m,m+1,\ldots \} \ni (n-1) \to n \in  \N \cap (m,\infty)$ note that \eqref{recursive_ineq:assumption1} assures that for all $n \in  \N \cap (m,\infty)$ with $e_{n-1} \leq C \gamma_{n-1}$ it holds that
\begin{equation}
\begin{split}
e_n 
&\leq
(1-\gamma_n)^2 e_{n-1} +  \kappa (\gamma_n)^2\\
&\leq
(1- 2\gamma_n + (\gamma_n)^2 ) C \gamma_{n-1} +  \kappa (\gamma_n)^2 - C \gamma_{n} + C \gamma_{n} \\
&=
C \! \left( \gamma_{n-1} - 2\gamma_n\gamma_{n-1}  + (\gamma_n)^2 \gamma_{n-1} - \gamma_{n}\right) +  \kappa (\gamma_n)^2 + C \gamma_{n} \\
&=
(\gamma_n)^2 \! \left[C \! \left( \frac{\gamma_{n-1}  - \gamma_{n}}{(\gamma_n)^2}- \frac{2\gamma_{n-1}}{\gamma_n}  +  \gamma_{n-1} \right) +  \kappa \right] + C \gamma_{n} \\	
&=
 C \gamma_{n} - (\gamma_n)^2 \! \left[C \! \left( \frac{\gamma_{n}  - \gamma_{n-1}}{(\gamma_n)^2} + \frac{2\gamma_{n-1}}{\gamma_n} -  \gamma_{n-1}\right) - \kappa \right].
\end{split}
\end{equation}
This, \eqref{recursive_ineq:eq1}, and the fact that $C \geq \frac{\kappa}{\mathcal{I}}$ demonstrate that for all $n \in \N \cap (m,\infty)$ with $e_{n-1} \leq C \gamma_{n-1}$ it holds that
\begin{equation}
\begin{split}
e_n 
&\leq
 C \gamma_{n} - (\gamma_n)^2 \left[C\mathcal{I} -  \kappa \right] \\
&\leq 
 C \gamma_{n} - (\gamma_n)^2 \left[ \tfrac{\kappa}{\mathcal{I}} \, \mathcal{I} -  \kappa \right] \\
 &=
C \gamma_{n} - (\gamma_n)^2 \left[\kappa -  \kappa \right] \\
&=
C \gamma_{n}. 
\end{split}
\end{equation}
Induction thus establishes \eqref{recursive_ineq:eq3}.
Next observe that \eqref{recursive_ineq:eq3} implies that for all $n \in \N$ it holds that
\begin{equation}
\begin{split}
e_n 
&\leq
\left[\max \! \left\{ \frac{e_1}{\gamma_1}, \frac{e_2}{\gamma_2},\ldots, \frac{e_{m-1}}{\gamma_{m-1}}, C\right\} \right] \gamma_n \\
&=
\left[\max \! \left\{ \frac{e_1}{\gamma_1}, \frac{e_2}{\gamma_2}, \ldots, \frac{e_{m}}{\gamma_{m}}, \frac{\kappa}{\mathcal{I}}\right\} \right] \gamma_n.
\end{split}
\end{equation}
This completes the proof of Lemma~\ref{recursive_ineq}.
\end{proof}

\subsubsection{On an asymptotic property of the learning rates}
\label{subsubsect:UB_small_aux}

\begin{lemma}
\label{aux_UB_nuleq1}
Let $\beta \in (0,\infty)$, $\nu \in (0,1)$, $(\gamma_n)_{n \in \N} \subseteq (0,\infty)$ satisfy for all $n \in \N$ that $\gamma_n = \beta n^{-\nu}$.
Then 
\begin{equation}
\liminf_{l \to \infty}\left[\frac{\gamma_l - \gamma_{l-1}}{(\gamma_l)^{2}} + \frac{2\gamma_{l-1}}{\gamma_{l}} - \gamma_{l-1} \right] \geq 2 > 0.
\end{equation}
\end{lemma}

\begin{proof}[Proof of Lemma~\ref{aux_UB_nuleq1}]
First, note that for all $l \in \{2,3,\ldots \}$ it holds that
\begin{equation}
\begin{split}
\gamma_l - \gamma_{l-1} 
&= 
\beta \, (l^{-\nu}- (l-1)^{-\nu}) 
= 
\beta \left[x^{-\nu} \right]_{x = l-1}^{x = l}
=
\beta \left[ \int_{l-1}^l (-\nu)x^{-\nu-1} \, \mathrm{d} x \right] \\
&=
-\beta \nu \underbrace{\left[ \int_{l-1}^l \frac{1}{x^{1+\nu}} \, \mathrm{d}x \right]}_{\leq \frac{1}{(l-1)^{1+\nu}}}
\geq
\frac{-\beta \nu}{(l-1)^{1+\nu}}.
\end{split}
\end{equation}
The fact that $ 1+\nu-2\nu = 1-\nu > 0$ hence demonstrates that
\begin{equation}
\begin{split}
\liminf_{l \to \infty} \left[\frac{\gamma_l - \gamma_{l-1}}{(\gamma_l)^{2}}\right]
&\geq
\liminf_{l \to \infty} \left[\frac{\big(\tfrac{-\beta \nu}{(l-1)^{1+\nu}}\big)}{\big(\tfrac{\beta}{l^\nu}\big)^{2}}\right]
=
-\tfrac{\nu}{\beta} \limsup_{l \to \infty} \left[\tfrac{l^{2\nu}}{(l-1)^{1+\nu}}\right] \\
&=
-\tfrac{\nu}{\beta} \limsup_{l \to \infty} \left[\tfrac{1}{(l-1)^{1+\nu-2\nu}}\left(\tfrac{l}{l-1}\right)^{2\nu} \right]
= 
0.
\end{split}
\end{equation}
Therefore, we obtain that
\begin{equation}
\begin{split}
&\liminf_{l \to \infty}\left[\frac{\gamma_l - \gamma_{l-1}}{(\gamma_l)^{2}} + \frac{2\gamma_{l-1}}{\gamma_{l}} - \gamma_{l-1} \right]  \\
&\geq
\liminf_{l \to \infty}\left[\frac{\gamma_l - \gamma_{l-1}}{(\gamma_l)^{2}}\right]
 + 2 \liminf_{l \to \infty} \left[ \frac{\gamma_{l-1}}{\gamma_{l}} \right]
  -  \limsup_{l \to \infty} \gamma_{l-1} \\
&\geq 
2 \liminf_{l \to \infty} \left[\frac{\big(\tfrac{\beta}{(l-1)^\nu}\big)}{\big(\tfrac{\beta}{l^\nu}\big)}\right] -  \limsup_{l \to \infty} \left[\tfrac{\beta}{(l-1)^\nu}\right]  \\
&=
2 \liminf_{l \to \infty} \Big[\! \left(\tfrac{l}{l-1} \right)^\nu \! \Big]
= 
2
.
\end{split}
\end{equation}
The proof of Lemma~\ref{aux_UB_nuleq1} is thus completed.
\end{proof}

\subsubsection{Upper error estimates}
\label{subsubsect:UB}

\begin{prop}
\label{UB_nuleq1}
Assume Setting~\ref{setting} and assume that $\nu < 1$.
Then there exists $C \in (0,\infty)$ such that for all $n \in \N$ it holds that
\begin{equation}
\big(\EXP{\norm{\Theta_n-\EXp{X_1}}^2}\big)^{\nicefrac{1}{2}} 
\leq 
 C n^{-\nicefrac{\nu}{2}}.
\end{equation}
\end{prop}

\begin{proof}[Proof of Proposition~\ref{UB_nuleq1}]
Note that items~\eqref{explicit_expressions:item3}--\eqref{explicit_expressions:item4} in Proposition~\ref{explicit_expressions} assure that for all $n \in \N$ it holds that
\begin{equation}
\label{UB_nuleq1:eq1}
\begin{split}
&\Exp{\norm{\Theta_n - \EXp{X_1}}^2} \\
&=
(1-\tfrac{\gamma\alpha}{n^\nu})^2 \, \Exp{\norm{\Theta_{n-1} - \EXp{X_1}}^2} + ( \tfrac{\gamma\alpha}{n^\nu})^2 \,\Exp{\norm{X_1 - \EXp{X_1}}^2} < \infty.
\end{split}
\end{equation}
Moreover, observe that Lemma~\ref{aux_UB_nuleq1} (with 
$\beta = \gamma \alpha$, $\nu = \nu$
in the notation of Lemma~\ref{aux_UB_nuleq1}) ensures that
\begin{equation}
\liminf_{l \to \infty}\left[\frac{\big(\tfrac{\gamma\alpha}{l^\nu}\big) - \big(\tfrac{\gamma\alpha}{(l-1)^\nu}\big)}{\big(\tfrac{\gamma\alpha}{l^\nu}\big)^{2}} + \frac{2\big(\tfrac{\gamma\alpha}{(l-1)^\nu}\big)}{\left(\tfrac{\gamma\alpha}{l^\nu}\right)} - \left(\tfrac{\gamma\alpha}{(l-1)^\nu}\right) \right] > 0.
\end{equation}
Combining this, (\ref{UB_nuleq1:eq1}) and Lemma~\ref{recursive_ineq} (with 
$\kappa = \Exp{\norm{X_1 - \EXp{X_1}}^2}$, 
$(e_n)_{n \in \N_0} = (\Exp{\norm{\Theta_n - \EXp{X_1}}^2})_{n \in \N_0}$,
$(\gamma_n)_{n \in \N} = (\tfrac{\gamma\alpha}{n^\nu})_{n \in \N}$ 
in the notation of Lemma~\ref{recursive_ineq}) establishes that there exists $C \in (0,\infty)$ such that for all $n \in \N$ it holds that
\begin{equation}
\Exp{\norm{\Theta_n - \EXp{X_1}}^2} \leq C (\tfrac{\gamma\alpha}{n^\nu}) = \left[C\gamma\alpha  \right] {n^{-\nu}}.
\end{equation}
Therefore, we obtain for all $n \in \N$ that
\begin{equation}
\left(\Exp{\norm{\Theta_n - \EXp{X_1}}^2}\right)^{\nicefrac{1}{2}} \leq \left[C\gamma\alpha  \right]^{\nicefrac{1}{2}} n^{-\nicefrac{\nu}{2}}.
\end{equation}
The proof of Proposition~\ref{UB_nuleq1} is thus completed.
\end{proof}

\subsection{Upper errors estimates in the case of fast decaying learning rates}
\label{subsect:UB_fast}
In this subsection we establish in Proposition~\ref{UB_nueq1_2} below an upper bound for the root mean square error of the SGD process in \eqref{setting:eq2} in the case of fast decaying learning rates (corresponding to the case $\nu= 1$ in Setting~\ref{setting}). 

\begin{prop}
\label{UB_nueq1_2}
Assume Setting~\ref{setting} and assume that $\nu = 1$. 
Then for every $\varepsilon \in (0,\infty)$ there exists $C \in (0,\infty)$ such that for all $n \in \N$ it holds that
\begin{equation}
\label{UB_nueq1_2:concl1}
\big(\EXP{\norm{\Theta_n-\EXp{X_1}}^2}\big)^{\nicefrac{1}{2}} 
\leq 
 C n^{- (\min \{\nicefrac{1}{2}, \gamma\alpha \} - \varepsilon)}.
\end{equation}
\end{prop}

\begin{proof}[Proof of Proposition~\ref{UB_nueq1_2}]
Throughout this proof let $\varepsilon \in (0,\min \{\nicefrac{1}{2}, \gamma\alpha \})$, let $\beta \in (0,\nicefrac{1}{2})$ be given by 
\begin{equation}
\beta = \min \{\nicefrac{1}{2}, \gamma\alpha \} - \varepsilon,
\end{equation}
and let $c = (c_n)_{n \in \N} \colon \N \to [0,\infty]$ be the function which satisfies for all $n \in \N$ that 
\begin{equation}
\label{UB_nueq1_2:eq1}
c_n = \max \! \left\{  \left\{ \tfrac{\EXp{\norm{\Theta_k-\EXp{X_1}}^2}}{k^{-2\beta}} \colon k \in \{1,2,\ldots, n\} \right\} \cup \left\{ (\gamma\alpha)^2 \,\EXP{\norm{X_1-\EXp{X_1}}^2} \right\} \! \right\}.
\end{equation}
Note that item~\eqref{explicit_expressions:item4} in Proposition~\ref{explicit_expressions} implies that for all $n \in \N$ it holds that
\begin{equation}
\EXP{\norm{\Theta_n-\EXp{X_1}}^2} < \infty.
\end{equation}
Hence, we obtain that for all $n \in \N$ it holds that
\begin{equation}
\label{UB_nueq1_2:eq2}
c_n < \infty.
\end{equation}
Next observe that
\begin{equation}
\begin{split}
&\limsup_{n \to \infty} \left[n^2 \! \left( (1-\tfrac{\gamma\alpha}{n})^2 (n-1)^{-2\beta} - n^{-2\beta}  \right) \right]\\
&=
\limsup_{n \to \infty} \left[n^2\! \left(  \Big[1-\tfrac{2\gamma\alpha}{n} + (\tfrac{\gamma\alpha}{n})^2\Big] (n-1)^{-2\beta} - n^{-2\beta}  \right)\right] \\
&=
\limsup_{n \to \infty} \left[n^2 ((n-1)^{-2\beta} - n^{-2\beta}) - n^2 (\tfrac{2\gamma\alpha}{n} )(n-1)^{-2\beta} + n^2(\tfrac{\gamma\alpha}{n})^2 (n-1)^{-2\beta}\right] \\
&=
\limsup_{n \to \infty}\left[ n^2 ((n-1)^{-2\beta} - n^{-2\beta})  - (2\gamma\alpha)n(n-1)^{-2\beta} + \tfrac{\left(\gamma\alpha\right)^2}{(n-1)^{2\beta}}\right] \\
&=
\limsup_{n \to \infty} \left[n^2 ((n-1)^{-2\beta} - n^{-2\beta})  - (2\gamma\alpha)n(n-1)^{-2\beta}\right].
\end{split}
\end{equation}
The fact that for all $n \in \{2,3,\ldots \}$ it holds that
\begin{equation}
\begin{split}
(n-1)^{-2\beta} - n^{-2\beta} 
&= 
- \left[x^{-2\beta} \right]_{x = n-1}^{x = n} 
= 
-\int_{n-1}^n (-2\beta) x^{-2\beta-1} \, \mathrm{d}x \\
&= 
2\beta \left[\int_{n-1}^n  x^{-2\beta-1} \, \mathrm{d}x \right]
\leq
2\beta(n-1)^{-2\beta-1}
\end{split}
\end{equation} 
hence proves that 
\begin{equation}
\label{UB_nueq1_2:eq3}
\begin{split}
&\limsup_{n \to \infty} \left[n^2 \!\left( (1-\tfrac{\gamma\alpha}{n})^2 (n-1)^{-2\beta} - n^{-2\beta}  \right) \right] \\
&\leq
\limsup_{n \to \infty}  \left[ n^2 2\beta(n-1)^{-2\beta-1}  - (2\gamma\alpha)n(n-1)^{-2\beta} \right]\\
&=
\limsup_{n \to \infty}  \left[n(n-1)^{-2\beta}( 2\beta n (n-1)^{-1}  - 2\gamma\alpha) \right] .
\end{split}
\end{equation}
Moreover, note that the fact that $2\beta = \min \{1, 2\gamma\alpha \} - 2\varepsilon \leq  2\gamma\alpha - 2\varepsilon< 2\gamma\alpha$ ensures that 
\begin{equation}
\limsup_{n \to \infty} \left[ 2\beta (\tfrac{n+1}{n}) - 2\gamma\alpha\right] = \liminf_{n \to \infty} \left[ 2\beta (\tfrac{n+1}{n}) - 2\gamma\alpha\right] = \ 2\beta  - 2\gamma\alpha < 0.
\end{equation}
The fact that $2\beta  < 2\gamma\alpha$, the fact that $2\beta  < 1$, and (\ref{UB_nueq1_2:eq3}) hence establish that
\begin{equation}
\begin{split}
&\limsup_{n \to \infty} \left[n^2 \! \left( (1-\tfrac{\gamma\alpha}{n})^2 (n-1)^{-2\beta} - n^{-2\beta}  \right) \right]\\
&\leq
\limsup_{n \to \infty} \left[\tfrac{n}{(n-1)^{2\beta}} \left( 2\beta (\tfrac{n}{n-1})  - 2\gamma\alpha\right)\right]  \\
&=
\limsup_{n \to \infty} \bigg[\! \left[\tfrac{(n+1)}{n^{2\beta}} \right] \left[ 2\beta (\tfrac{n+1}{n})  - 2\gamma\alpha\right] \! \bigg]  \\
&=
\lim_{n \to \infty} \bigg[\! \left[\tfrac{(n+1)}{n^{2\beta}} \right] \left[ 2\beta (\tfrac{n+1}{n})  - 2\gamma\alpha\right] \! \bigg]  \\
&=
\left[\lim_{n \to \infty}  \left[\tfrac{(n+1)}{n^{2\beta}} \right]  \right] \left[\lim_{n \to \infty} \left[ 2\beta (\tfrac{n+1}{n})  - 2\gamma\alpha\right] \right]   \\
&=
\left[\lim_{n \to \infty}  \left[n^{1-2\beta} +n^{-2\beta} \right]  \right] \left[ 2\beta  - 2\gamma\alpha \right]   \\
&=
\left[\lim_{n \to \infty}  n^{1-2\beta} \right] \left[ 2\beta  - 2\gamma\alpha\right] 
= 
-\infty.
\end{split}
\end{equation}
This implies that there exists $m \in \N$ such that for all $n \in \N \cap (m,\infty)$ it holds that
\begin{equation}
\label{UB_nueq1_2:eq4}
n^2 \! \left( (1-\tfrac{\gamma\alpha}{n})^2 (n-1)^{-2\beta} - n^{-2\beta}  \right) \leq -1.
\end{equation}
Next note that \eqref{UB_nueq1_2:eq2} establishes that $c_m < \infty$.
Moreover, observe that \eqref{UB_nueq1_2:eq1} ensures that for all $n \in \{1,2,\ldots, m \}$ it holds that
\begin{equation}
\label{UB_nueq1_2:eq5}
\begin{split}
\Exp{\norm{\Theta_n - \EXp{X_1}}^2} 
= 
\left[\tfrac{\EXp{\norm{\Theta_n-\EXp{X_1}}^2}}{n^{-2\beta}}\right] n^{-2\beta}
\leq 
c_m n^{-2\beta}.
\end{split}
\end{equation}
Next we claim that for all $n \in \{m,m+1,\ldots \}$ it holds that
\begin{equation}
\label{UB_nueq1_2:eq6}
\EXP{\norm{\Theta_n-\EXp{X_1}}^2} \leq c_m n^{-2\beta}.
\end{equation} 
Note that \eqref{UB_nueq1_2:eq5} establishes (\ref{UB_nueq1_2:eq6}) in the base case $n = m$.
For the induction step $ \{m,m+1,\ldots \} \ni (n-1) \to n \in  \N \cap (m,\infty)$ note that item~\eqref{explicit_expressions:item3} in Proposition~\ref{explicit_expressions} assures that for all $n \in  \{2,3,\ldots \}$ with $\EXP{\norm{\Theta_{n-1}-\EXp{X_1}}^2} \leq c_m (n-1)^{-2\beta}$ it holds that
\begin{equation}
\begin{split}
&\Exp{\norm{\Theta_n - \EXp{X_1}}^2} \\
&=
(1-\tfrac{\gamma\alpha}{n})^2 \, \Exp{\norm{\Theta_{n-1} - \EXp{X_1}}^2} + ( \tfrac{\gamma\alpha}{n})^2 \,\Exp{\norm{X_1 - \EXp{X_1}}^2} \\
&\leq 
(1-\tfrac{\gamma\alpha}{n})^2 \, c_m (n-1)^{-2\beta} + ( \tfrac{\gamma\alpha}{n})^2 \, \Exp{\norm{X_1 - \EXp{X_1}}^2} \\
&=
(1-\tfrac{\gamma\alpha}{n})^2 \, c_m (n-1)^{-2\beta} - c_m n^{-2\beta} + ( \tfrac{\gamma\alpha}{n})^2 \,\Exp{\norm{X_1 - \EXp{X_1}}^2} + c_m n^{-2\beta}\\
&=
\tfrac{1}{n^2} \Big[ c_m\left[ n^2 \! \left( (1-\tfrac{\gamma\alpha}{n})^2 (n-1)^{-2\beta} - n^{-2\beta}  \right)\right] + ( \gamma\alpha )^2 \, \Exp{\norm{X_1 - \EXp{X_1}}^2} \Big] \\
&\quad + c_m n^{-2\beta}.
\end{split}
\end{equation}
This, \eqref{UB_nueq1_2:eq1}, and \eqref{UB_nueq1_2:eq4} demonstrate that for all $n \in  \N \cap (m,\infty)$ with $\EXP{\norm{\Theta_{n-1}-\EXp{X_1}}^2} \leq c_m (n-1)^{-2\beta}$ it holds that
\begin{equation}
\begin{split}
&\Exp{\norm{\Theta_n - \EXp{X_1}}^2} \\
&\leq 
\tfrac{1}{n^2} \Big[ -c_m  + ( \gamma\alpha )^2 \, \Exp{\norm{X_1 - \EXp{X_1}}^2} \Big] + c_m n^{-2\beta} \\
&\leq 
\tfrac{1}{n^2} \Big[ -( \gamma\alpha )^2\, \Exp{\norm{X_1 - \EXp{X_1}}^2} +  ( \gamma\alpha )^2\,\Exp{\norm{X_1 - \EXp{X_1}}^2} \Big] + c_m n^{-2\beta} \\
&=
c_m n^{-2\beta}.
\end{split}
\end{equation}
Induction thus proves \eqref{UB_nueq1_2:eq6}.
Combining \eqref{UB_nueq1_2:eq5} and \eqref{UB_nueq1_2:eq6} establishes \eqref{UB_nueq1_2:concl1}.
The proof of Proposition~\ref{UB_nueq1_2} is thus completed.
\end{proof}

\subsection{Refined upper errors estimates in the case of fast decaying learning rates}
\label{subsect:UB_fast_refined}

\subsubsection[On an asymptotic property for fast decaying learning rates]{A characterization of an asymptotic property for fast decaying learning rates}
\label{subsubsect:UB_fast_aux}

\begin{lemma}
\label{aux_UB_nueq1}
Let $\beta \in (0,\infty)$, $(\gamma_n)_{n \in \N} \subseteq (0,\infty)$ satisfy for all $n \in \N$ that $\gamma_n = \nicefrac{\beta}{n}$. 
Then the following two statements are equivalent:
\begin{enumerate}[(i)]
\item \label{aux_UB_nueq1:item1}
It holds that
\begin{equation}
\liminf_{l \to \infty}\left[\frac{\gamma_l - \gamma_{l-1}}{(\gamma_l)^{2}} + \frac{2\gamma_{l-1}}{\gamma_{l}} - \gamma_{l-1} \right] > 0.
\end{equation}
\item \label{aux_UB_nueq1:item2}
It holds that $\beta > \nicefrac{1}{2}$.
\end{enumerate}
\end{lemma}

\begin{proof}[Proof of Lemma~\ref{aux_UB_nueq1}]
First, observe that 
\begin{equation}
\begin{split}
&\liminf_{l \to \infty}\left[\frac{\gamma_l - \gamma_{l-1}}{(\gamma_l)^{2}} + \frac{2\gamma_{l-1}}{\gamma_{l}} - \gamma_{l-1} \right] \\
&=
\liminf_{l \to \infty}\left[\frac{\left(\tfrac{\beta}{l}\right)- \left(\tfrac{\beta}{l-1}\right)}{\left(\tfrac{\beta}{l}\right)^{2}} + \frac{2\left(\tfrac{\beta}{l-1}\right)}{\left(\tfrac{\beta}{l}\right)} - \left(\tfrac{\beta}{l-1}\right) \right] \\
&=
\liminf_{l \to \infty}\left[\frac{1}{\beta} \left[\frac{\left(\tfrac{l-1-l}{l(l-1)}\right)}{\left(\tfrac{1}{l}\right)^{2}} \right] + 2 \left(\tfrac{l}{l-1}\right) - \beta\left(\tfrac{1}{l-1}\right) \right] \\
&=
\liminf_{l \to \infty}\left[-\tfrac{1}{\beta} \left(\tfrac{l^2}{l(l-1)}\right) + 2 \left(\tfrac{l}{l-1}\right) - \beta\left(\tfrac{1}{l-1}\right) \right] \\
&=
\liminf_{l \to \infty}\left[-\tfrac{1}{\beta} \left(\tfrac{l}{l-1}\right) + 2 \left(\tfrac{l}{l-1}\right) - \beta\left(\tfrac{1}{l-1}\right) \right] \\
&=
\lim_{l \to \infty}\left[-\tfrac{1}{\beta} \left(\tfrac{l}{l-1}\right) + 2 \left(\tfrac{l}{l-1}\right) - \beta\left(\tfrac{1}{l-1}\right) \right] \\
&=
-\tfrac{1}{\beta}  + 2 
= 2 - \tfrac{1}{\beta}.
\end{split}
\end{equation}
Therefore, we obtain that
\begin{equation}
\begin{split}
&\qquad 
\left( \liminf_{l \to \infty}\left[\frac{\gamma_l - \gamma_{l-1}}{(\gamma_l)^{2}} + \frac{2\gamma_{l-1}}{\gamma_{l}} - \gamma_{l-1} \right] > 0 \right) \\
\Leftrightarrow
&\qquad   
\left( 2 - \tfrac{1}{\beta} > 0 \right) \\
\Leftrightarrow
&\qquad   
\left(2 > \tfrac{1}{\beta} \right)\\
\Leftrightarrow
&\qquad   
\left(\tfrac{1}{2} < \beta \right).
\end{split}
\end{equation}
The proof of Lemma~\ref{aux_UB_nueq1} is thus completed.
\end{proof}

\subsubsection[Error estimates for large but fast decaying learning rates]{Improved upper error estimates in the case of large but fast decaying learning rates}
\label{subsubsect:UB_large_and_fast_refined}

\begin{prop}
\label{UB_nueq1_1}
Assume Setting~\ref{setting} and assume that $\nu = 1$ and $\gamma \alpha > \nicefrac{1}{2}$.
Then there exists $C \in (0,\infty)$ such that for all $n \in \N$ it holds that
\begin{equation}
\big(\EXP{\norm{\Theta_n-\EXp{X_1}}^2}\big)^{\nicefrac{1}{2}} 
\leq 
 C n^{-\nicefrac{1}{2}}.
\end{equation}
\end{prop}

\begin{proof}[Proof of Proposition~\ref{UB_nueq1_1}]
Observe that items~\eqref{explicit_expressions:item3}--\eqref{explicit_expressions:item4} in Proposition~\ref{explicit_expressions} imply that for all $n \in \N$ it holds that
\begin{equation}
\label{UB_nueq1_1:eq1}
\begin{split}
&\Exp{\norm{\Theta_n - \EXp{X_1}}^2} \\
&=
(1-\tfrac{\gamma\alpha}{n})^2 \, \Exp{\norm{\Theta_{n-1} - \EXp{X_1}}^2} + ( \tfrac{\gamma\alpha}{n})^2 \, \Exp{\norm{X_1 - \EXp{X_1}}^2}
< 
\infty.
\end{split}
\end{equation}
Moreover, note that the hypothesis that $\gamma \alpha > \nicefrac{1}{2}$ and Lemma~\ref{aux_UB_nueq1} (with 
$\beta = \gamma \alpha$ 
in the notation of Lemma~\ref{aux_UB_nueq1}) 
ensures that
\begin{equation}
\liminf_{l \to \infty}\left[\frac{\left(\tfrac{\gamma\alpha}{l}\right) - \big(\tfrac{\gamma\alpha}{(l-1)}\big)}{\left(\tfrac{\gamma\alpha}{l}\right)^{2}} + \frac{2\big(\tfrac{\gamma\alpha}{(l-1)}\big)}{\left(\tfrac{\gamma\alpha}{l}\right)} - \left(\tfrac{\gamma\alpha}{(l-1)}\right) \right] > 0.
\end{equation}
Combining this and \eqref{UB_nueq1_1:eq1} with Lemma~\ref{recursive_ineq} (with 
$\kappa = \Exp{\norm{X_1 - \EXp{X_1}}^2}$, 
$(e_n)_{n \in \N_0} = (\Exp{\norm{\Theta_n - \EXp{X_1}}^2})_{n \in \N_0}$,
$(\gamma_n)_{n \in \N} = (\tfrac{\gamma\alpha}{n})_{n \in \N}$ 
in the notation of Lemma~\ref{recursive_ineq}) demonstrates that there exists $C \in (0,\infty)$ such that for all $n \in \N$ it holds that
\begin{equation}
\Exp{\norm{\Theta_n - \EXp{X_1}}^2} \leq C (\tfrac{\gamma\alpha}{n}) = \left[C\gamma\alpha  \right] {n^{-1}}.
\end{equation}
Therefore, we obtain for all $n \in \N$ that
\begin{equation}
\left(\Exp{\norm{\Theta_n - \EXp{X_1}}^2}\right)^{\nicefrac{1}{2}} \leq \left[C\gamma\alpha  \right]^{\nicefrac{1}{2}} n^{-\nicefrac{1}{2}}.
\end{equation}
The proof of Proposition~\ref{UB_nueq1_1} is thus completed.
\end{proof}

\subsubsection[Error estimates in the case of fast decaying learning rates]{Refined upper error estimates in the case of fast decaying learning rates}
\label{subsubsect:UB_fast_refined}

The next result, Corollary~\ref{UB_nueq1_3} below, combines the upper error bounds for fast decaying learning rates obtained in Proposition~\ref{UB_nueq1_2} and Proposition~\ref{UB_nueq1_1} above.

\begin{cor}
\label{UB_nueq1_3}
Assume Setting~\ref{setting} and assume that $\nu = 1$. 
Then for every $\varepsilon \in (0,\infty)$ there exists $C \in (0,\infty)$ such that for all $n \in \N$ it holds that
\begin{equation}
\label{UB_nueq1_3:conclusion}
\big(\EXP{\norm{\Theta_n-\EXp{X_1}}^2}\big)^{\nicefrac{1}{2}} 
\leq 
 C n^{- \min \{\nicefrac{1}{2}, \gamma\alpha - \varepsilon \} }.
\end{equation}
\end{cor}

\begin{proof}[Proof of Corollary~\ref{UB_nueq1_3}]
To prove \eqref{UB_nueq1_3:conclusion} we distinguish between the cases $\gamma \alpha > \nicefrac{1}{2}$ and $\gamma \alpha \leq \nicefrac{1}{2}$.
We first consider the case $\gamma \alpha > \nicefrac{1}{2}$.
In this case observe that Proposition~\ref{UB_nueq1_1} proves that there exists $C \in (0,\infty)$ such that for all $\varepsilon \in (0,\infty)$, $n \in \N$ it holds that
\begin{equation}
\begin{split}
\big(\EXP{\norm{\Theta_n-\EXp{X_1}}^2}\big)^{\nicefrac{1}{2}} 
&\leq
C n^{-\nicefrac{1}{2}} \\
&=
C n^{[\min \{\nicefrac{1}{2}, \gamma\alpha - \varepsilon \}-\nicefrac{1}{2}]} n^{-\min \{\nicefrac{1}{2}, \gamma\alpha - \varepsilon \}} \\
&=
C n^{\min \{0, \gamma\alpha - \varepsilon-\nicefrac{1}{2} \}} n^{-\min \{\nicefrac{1}{2}, \gamma\alpha - \varepsilon \}} \\
&\leq 
C n^{-\min \{\nicefrac{1}{2}, \gamma\alpha - \varepsilon \}}.
\end{split}
\end{equation}
Hence, we obtain that for every $\varepsilon \in (0,\infty)$ there exists $C \in (0,\infty)$ such that for all $n \in \N$ it holds that 
\begin{equation}
\label{UB_nueq1_3:eq1}
\begin{split}
\big(\EXP{\norm{\Theta_n-\EXp{X_1}}^2}\big)^{\nicefrac{1}{2}} 
\leq 
C n^{-\min \{\nicefrac{1}{2}, \gamma\alpha - \varepsilon \}}.
\end{split}
\end{equation}
This establishes \eqref{UB_nueq1_3:conclusion} in the case $\gamma \alpha > \nicefrac{1}{2}$.
Next we consider the case  $\gamma \alpha \leq \nicefrac{1}{2}$.
In this case we observe that Proposition \ref{UB_nueq1_2} proves that for every $\varepsilon \in (0,\infty)$ there exists $C \in (0,\infty)$ such that for all $n \in \N$ it holds that
\begin{equation}
\label{UB_nueq1_3:eq2}
\begin{split}
\big(\EXP{\norm{\Theta_n-\EXp{X_1}}^2}\big)^{\nicefrac{1}{2}} 
&\leq
C n^{- (\min \{\nicefrac{1}{2}, \gamma\alpha  \} - \varepsilon)}
= 
C n^{-(\gamma\alpha - \varepsilon)}
=
C n^{-\min \{\nicefrac{1}{2}, \gamma\alpha - \varepsilon \}}.
\end{split}
\end{equation}
This proves \eqref{UB_nueq1_3:conclusion} in the case $\gamma \alpha \leq \nicefrac{1}{2}$.
Combining \eqref{UB_nueq1_3:eq1} and \eqref{UB_nueq1_3:eq2} establishes \eqref{UB_nueq1_3:conclusion}.
The proof of Corollary~\ref{UB_nueq1_3} is thus completed.
\end{proof}

\subsection{Upper error estimates in the case of very fast decaying learning rates}
\label{subsect:UB_very_fast}

In this subsection we establish in Lemma~\ref{UB_very_fast} below that the root mean square error of the SGD process in \eqref{setting:eq2} is bounded from above in the case of very fast decaying learning rates (corresponding to the case $\nu > 1$ in Setting~\ref{setting}).

\begin{lemma}
\label{UB_very_fast}
Assume Setting~\ref{setting} and assume that $\nu > 1$. 
Then there exists $C \in (0,\infty)$ such that for all $n \in \N$ it holds that
\begin{equation}
\label{UB_very_fast:concl1}
	\big(\EXP{\norm{\Theta_n-\EXp{X_1}}^2}\big)^{\nicefrac{1}{2}} 
\leq 
	 C.
\end{equation}
\end{lemma}

\begin{proof}[Proof of Lemma~\ref{UB_very_fast}]
Throughout this proof let $m \in \N \cap (\gamma\alpha, \infty)$ and let $C \in (0,\infty)$ be given by
\begin{equation}
\label{UB_very_fast:eq1}
	C
= 
\max \! 
\left(
	\left[
	\bigcup_{\substack{k,r \, \in \,\N, \\ r \leq k \leq m}}
	\left\{ 
		 {\textstyle\prod\limits_{l = r}^k }\Abs{1-\tfrac{\gamma\alpha}{l^\nu}}^2
	\right\}
	\right] \cup 
	\left\{ 
		1
	\right\}
\right).
\end{equation}
Observe that item~\eqref{explicit_expressions:item4} in Proposition~\ref{explicit_expressions} ensures that for all $n \in \N$ it holds that
\begin{equation}
\begin{split}
	&\Exp{\norm{\Theta_n - \EXp{X_1}}^2} \\
&=
	\norm{\xi-\EXp{X_1}}^2 \left[ \prod_{l = 1}^n (1-\tfrac{\gamma\alpha}{l^\nu})\right]^2  \\
&\quad + 
	\Exp{\norm{X_1 - \EXp{X_1}}^2}
	\left[
		\sum_{k = 1}^n 
		\left[
			\tfrac{\gamma\alpha}{k^\nu}\left( \prod_{l = k+1}^n (1-\tfrac{\gamma\alpha}{l^\nu})\right)
		\right]^2
	\right] \\
&=
	\norm{\xi-\EXp{X_1}}^2 
	\left[ \prod_{l = 1}^{\min\{m,n \}} \Abs{1-\tfrac{\gamma\alpha}{l^\nu}}^2 \right] 
	\left[ \prod_{l = \min\{m,n \}+1}^n \Abs{1-\tfrac{\gamma\alpha}{l^\nu}}^2 \right] \\
&\quad + 
	(\gamma\alpha)^2 \,
	\Exp{\norm{X_1 - \EXp{X_1}}^2} \\
&\quad \cdot	
	\left[
		\sum_{k = 1}^n 
		\left(
			\frac{1}{k^{2\nu}}
			\left[ \prod_{l = k+1}^{\min\{m,n \}} \Abs{1-\tfrac{\gamma\alpha}{l^\nu}}^2 \right] 
			\left[ \prod_{l = \max\{k,\min\{m,n \}\}+1}^n \Abs{1-\tfrac{\gamma\alpha}{l^\nu}}^2 \right]
		\right)
	\right] \\
\end{split}
\end{equation}
This and \eqref{UB_very_fast:eq1} establish that  for all $n \in \N$ it holds that
\begin{equation}
\label{UB_very_fast:eq2}
\begin{split}
	&\Exp{\norm{\Theta_n - \EXp{X_1}}^2} \\
&\leq
	C \norm{\xi-\EXp{X_1}}^2 
	\left[ \prod_{l = \min\{m,n \}+1}^n \Abs{1-\tfrac{\gamma\alpha}{l^\nu}}^2 \right] \\
&\quad + 
	(\gamma\alpha)^2 \,
	\Exp{\norm{X_1 - \EXp{X_1}}^2}	
	\left[
		\sum_{k = 1}^n 
		\left(
			\frac{C}{k^{2\nu}}
			\left[ \prod_{l = \max\{k,\min\{m,n \}\}+1}^n \Abs{1-\tfrac{\gamma\alpha}{l^\nu}}^2 \right]
		\right)
	\right].
\end{split}
\end{equation}
Next note that the fact that $m > \gamma\alpha$ assures that for all $ l \in \N \cap (m,\infty)$ it holds that
\begin{equation}
	0 
= 
	1 - 1 
< 
	1 - \tfrac{\gamma\alpha}{m} 
\leq 
	1 - \tfrac{\gamma\alpha}{m^\nu} 
\leq 
	1 - \tfrac{\gamma\alpha}{l^\nu} 
< 
	1-0 
= 
	1.
\end{equation}
Hence, we obtain that for all $ l \in \N \cap (m,\infty)$  it holds that
\begin{equation}
\left(1-\tfrac{\gamma\alpha }{l^\nu}\right) \in (0,1).
\end{equation}
This implies that for all $n,l \in \N$ with  $\min\{m,n \}< l \leq n$ it holds that
\begin{equation}
\Abs{1-\tfrac{\gamma\alpha}{l^\nu}}^2 = \left(1-\tfrac{\gamma\alpha}{l^\nu}\right)^2 \in (0,1).
\end{equation}
Combining this with \eqref{UB_very_fast:eq2} demonstrates that for all $n \in \N$ it holds that
\begin{equation}
\label{UB_very_fast:eq3}
\begin{split}
	&\Exp{\norm{\Theta_n - \EXp{X_1}}^2} \\
&\leq
	C \norm{\xi-\EXp{X_1}}^2 
	+
	(\gamma\alpha)^2 
	C \,
	\Exp{\norm{X_1 - \EXp{X_1}}^2}	
	\left[
		\sum_{k = 1}^n \frac{1}{k^{2\nu}}
	\right].
\end{split}
\end{equation}
Moreover, note that the hypothesis that $\nu > 1$ ensures that
\begin{equation}
\begin{split}
	\sum_{k = 1}^\infty  \frac{1}{k^{2\nu}} 
&= 
	1 + \sum_{k = 2}^\infty  \left[ \int_{k-1}^k \frac{1}{k^{2\nu}} \, \mathrm{d}x \right]
\leq  
	1 + \sum_{k = 2}^\infty \left[ \int_{k-1}^k \frac{1}{x^{2\nu}} \, \mathrm{d}x \right]
= 
	1 + \int_{1}^\infty \frac{1}{x^{2\nu}} \, \mathrm{d}x \\
&= 
	1 + \left[ \left(\tfrac{1}{(1-2\nu)}\right) x^{1-2\nu} \right]_{x = 1}^{x = \infty}
=
	1 - \left(\tfrac{1}{(1-2\nu)}\right) 
=
	1 + \tfrac{1 }{(2\nu-1)}
< 
	\infty.
\end{split}
\end{equation}
This and \eqref{UB_very_fast:eq3} establish \eqref{UB_very_fast:concl1}. 
The proof of Lemma~\ref{UB_very_fast} is thus completed.
\end{proof}

\section{Lower error estimates for the SGD optimization method}
\label{sect:LB}
In this section we establish in Proposition~\ref{LB} and Proposition~\ref{LB_nueq1} below lower bounds for the root mean square distance between the SGD process in \eqref{setting:eq2} and the global minimum of the considered optimization problem (cf.\ item~\eqref{properties_loss:item2} in Lemma~\ref{properties_loss}).
These results show that the upper error bounds obtained in Section~\ref{sect:UB} (see Proposition~\ref{UB_nuleq1} and Corollary~\ref{UB_nueq1_3} above) can essentially not be improved.
Moreover, in Subsection~\ref{subsect:LB_very_fast} below we demonstrate that the SGD process fails to converge to the global minimum of the objective function in the case of very fast decaying learning rates (see Lemma~\ref{LB_nubigger1} below for details).
Finally, in Subsection~\ref{subsect:main_result} below we present Theorem~\ref{main_theorem} which combines the main findings of this article.

\subsection[Lower errors estimates]{Lower errors estimates in the case of slowly and fast decaying learning rates}
\label{subsect:LB_slow_and_fast}

In this subsection we establish in Proposition~\ref{LB} below a lower bound for the root mean square error of the SGD process in \eqref{setting:eq2} in the case of slowly and fast decaying learning rates (corresponding to the case $\nu \leq 1$ in Setting~\ref{setting}).
Our proof of Proposition~\ref{LB} employs the elementary result in Lemma~\ref{positivity} and the elementary and well-known result in Lemma~\ref{fundamental_limit}. For completeness we also provide the proofs of Lemma~\ref{positivity} and Lemma~\ref{fundamental_limit} here. 

\subsubsection{On the strict positivity of the mean square errors}
\label{subsubsect:positivity}

\begin{lemma}
\label{positivity}
Assume Setting~\ref{setting} and assume that $\EXP{\norm{X_1-\EXp{X_1}}^2} > 0$.
Then it holds for all $n \in \N$ that
\begin{equation}
\EXP{\norm{\Theta_n-\EXp{X_1}}^2} > 0.
\end{equation}
\end{lemma}

\begin{proof}[Proof of Lemma~\ref{positivity}]
Observe that item~\eqref{explicit_expressions:item4} in Proposition~\ref{explicit_expressions} and the assumption that $\EXP{\norm{X_1-\EXp{X_1}}^2} > 0$ assure that for all $n \in \N$ it holds that
\begin{equation}
\begin{split}
&\Exp{\norm{\Theta_n - \EXp{X_1}}^2} \\
&= 
\left[ \prod_{l = 1}^n (1-\tfrac{\gamma\alpha}{l^\nu})\right]^2 \norm{\xi-\EXp{X_1}}^2 \\
& \quad +  
\Exp{\norm{X_1 - \EXp{X_1}}^2}
\left[
\sum_{k = 1}^n 
\left[
\tfrac{\gamma\alpha}{k^\nu}\left( \prod_{l = k+1}^n (1-\tfrac{\gamma\alpha}{l^\nu})\right)
\right]^2
\right]\\
& \geq
\Exp{\norm{X_1 - \EXp{X_1}}^2}  \left(\tfrac{\gamma\alpha}{n^\nu} \right) > 0.
\end{split} 
\end{equation}
The proof of Lemma~\ref{positivity} is thus completed.
\end{proof}

\subsubsection{Approximations of the exponential function}
\label{subsubsect:approx_exp}

\begin{lemma}
\label{lemma_approx_exp}
Let $ (a_l)_{l \in \N} \subseteq \R $,  $(n_l)_{l \in \N} \subseteq \N $ satisfy that 
$\liminf_{l \to \infty} a_l = \limsup_{l \to \infty} a_l$ and 
$\liminf_{ l \to \infty } n_l = \infty$.
Then
\begin{equation}
\limsup_{ l \to \infty } \ABs{ \!
    \left[ 
      1 +\tfrac{a_l}{ n_l }
    \right]^{ n_l }
    -
    \exp\!
    	\left(
    		\lim_{ l \to \infty }a_l
    	\right) \!
    }
= 0.
\end{equation}
\end{lemma}

\begin{proof}[Proof of Lemma~\ref{lemma_approx_exp}]
Throughout this proof let $ f_l \colon \N_0 \to \R $, $ l \in \N $, be the functions which satisfy for all $ l \in \N $, $ k \in \N_0 $ that
\begin{equation}
  f_l( k )
  =
  \begin{cases}
  	\left[ \prod\limits_{r = 0}^{k-1} (n_l-r) \right]
    \frac{ ( a_l )^k }{ ( n_l )^k k! }
  &
    \colon
    k \leq n_l
  \\
    0
  &
    \colon
    k > n_l,
  \end{cases}
\end{equation}
let $F \colon \N_0 \to \R$ be the function which satisfies for all $k \in \N_0$ that
\begin{equation}
F(k) = \frac{\left[ \sup_{ l \in \N } | a_l | \right]^k}{ k! },
\end{equation}
and let $\# \colon \mathcal{P}(\N_0) \to [0,\infty]$ be the counting measure on $\N_0$.
Observe that the binomial theorem proves that for all $l \in \N$ it holds that
\begin{equation}
\label{lemma_approx_exp:eq1}
    \left[ 
      1 
      +
      \tfrac{
        a_l
      }{ n_l }
    \right]^{ n_l }
  =
  \sum_{ k = 0 }^{ n_l }
  	\binom{n_l}{k}
  \left[ 
    \tfrac{ a_l }{ n_l }
  \right]^k
  =
  \sum_{ k = 0 }^{ \infty }
  f_l( k )
  =
  \int_{ \N_0 }
  f_l( k )
  \,
  \#( \mathrm{d} k )
  .
\end{equation}
Moreover, note the hypothesis that $\liminf_{ l \to \infty } n_l = \infty$ ensures that for all $ k \in \N_0 $ it holds that
\begin{equation}
\label{lemma_approx_exp:eq2}
  \lim_{ l \to \infty }
  f_l( k )
  =
  \lim_{ l \to \infty }  \left[ \left[ \textstyle{\prod\limits_{r = 0}^{k-1} (1-\tfrac{r}{n_l})} \right]  \tfrac{ ( a_l )^k }{ k! }\right]
  =
  \frac{
    \left[ \lim_{ l \to \infty } a_l \right]^k}{ k! }.
\end{equation}
In addition, note that for all $k \in \N_0$ it holds that
\begin{equation}
\sup_{ l \in \N }
  \left| f_l( k ) \right|
  \leq
  F(k).
\end{equation}
The fact that
\begin{equation}
	\int_{ \N_0 } F( k ) \, \#( \mathrm{d} k )
= 
	\sum_{k = 0}^{\infty}  \left[\frac{\left[ \sup_{ l \in \N } | a_l | \right]^k}{ k! }\right]
=
	\exp \! \left(\sup_{ l \in \N } | a_l |  \right)
< 
	\infty,
\end{equation}
 Lebesgue's theorem of dominated convergence, and \eqref{lemma_approx_exp:eq2} hence demonstrate that
\begin{equation}
\begin{split}
&
  \lim_{ l \to \infty }
  \left[
  \int_{ \N_0 }
  f_l( k )
  \,
  \#( \mathrm{d} k )
  \right] 
  =
  \int_{ \N_0 }
  \big[\lim\nolimits_{ l \to \infty }
  f_l( k ) \big]
  \,
  \#( \mathrm{d} k ) \\
&  =
  \int_{ \N_0 }
  \frac{
    \left[ \lim_{ l \to \infty } a_l \right]^k
  }{ k! }
  \,
  \#( \mathrm{d} k )
  =
  \sum_{ k = 0 }^{ \infty }
  \frac{
    \left[ 
      \lim_{ l \to \infty } a_l 
    \right]^k
  }{ k! }
  =
  \exp\!\left(
    \lim_{ l \to \infty } a_l 
  \right)
  .
\end{split}
\end{equation}
This and \eqref{lemma_approx_exp:eq1} ensure that
\begin{equation}
  \lim_{ l \to \infty }
  \Big[ 
    \left[ 
      1 
      +
      \tfrac{
        a_l
      }{ n_l }
    \right]^{ n_l }
  \Big]
  =
  \exp\!\left(
    \lim_{ l \to \infty }
    a_l
  \right)
  .
\end{equation}
The proof of Lemma~\ref{lemma_approx_exp} is thus completed.
\end{proof}

\begin{lemma}
\label{fundamental_limit}
Let $ (a_l)_{l \in \N} \subseteq \R $,  $(n_l)_{l \in \N} \subseteq \R $ satisfy that 
$\liminf_{l \to \infty} a_l = \limsup_{l \to \infty} a_l$ and 
$\liminf_{ l \to \infty } n_l = \infty$.
\begin{equation}
\limsup_{ l \to \infty } \ABs{ \!
    \left[ 
      1 +\tfrac{a_l}{ n_l }
    \right]^{ n_l }
    -
    \exp\!
    	\left(
    		\lim_{ l \to \infty }a_l
    	\right) \!
    }
= 0.
\end{equation}
\end{lemma}

\begin{proof}[Proof of Lemma~\ref{fundamental_limit}]
Throughout this proof let $\floor{\cdot} \colon \R \to \Z$ be the function which satisfies for all $x \in \R$ that $\floor{x} = \max( (-\infty,x] \cap  \Z)$.  
Observe that the hypothesis that $\liminf_{ l \to \infty } n_l = \infty$ ensures that
\begin{equation}
	1 
\geq 
	\limsup_{ l \to \infty } \left[ \tfrac{\floor{n_l} }{n_l} \right]
\geq
	\liminf_{ l \to \infty } \left[ \tfrac{\floor{n_l} }{n_l} \right]
\geq
	\liminf_{ l \to \infty } \left[ \tfrac{(n_l - 1)}{n_l} \right]
=
	\liminf_{ l \to \infty } \left[ 1 - \tfrac{1}{n_l} \right] 
= 
	1.
\end{equation}
This implies that 
$
\lim_{ l \to \infty } \big[ \tfrac{\floor{n_l} }{n_l} \big] = 1
$.
The hypothesis that $\liminf_{l \to \infty} a_l = \limsup_{l \to \infty} a_l$ therefore assures that
\begin{equation}
	\lim_{ l \to \infty } \left[ \tfrac{a_l \floor{n_l} }{n_l} \right]
=
	\left[\lim_{ l \to \infty } a_l\right] \left[ \lim_{ l \to \infty } \tfrac{ \floor{n_l} }{n_l} \right]
=
	\lim_{ l \to \infty } a_l.
\end{equation}
This, the fact that $\liminf_{ l \to \infty } \floor{n_l} = \infty$, and Lemma~\ref{lemma_approx_exp} 
(with
$(a_l)_{l \in \N} = (\tfrac{a_l \floor{n_l} }{n_l})_{l \in \N}$,
$(n_l)_{l \in \N} = (\floor{n_l})_{l \in \N}$
in the notation of Lemma~\ref{lemma_approx_exp})
proves that
\begin{equation}
\label{fundamental_limit:eq2}
	\lim_{ l \to \infty }  \left( \left[ 1 + \tfrac{ \left(\frac{a_l\floor{n_l} }{n_l}\right)}{\floor{n_l} }  \right]^{ \floor{n_l} } \right)
= 
	\exp \! \left( \lim_{ l \to \infty }  \left[ \tfrac{a_l\floor{n_l} }{n_l} \right]\right) 
=
	\exp \! \left( \lim_{ l \to \infty }  a_l\right).
\end{equation}
Next note that 
the fact that for all $\alpha \in [0,1], r \in (0,1]$ it holds that $r \leq r^\alpha \leq 1$ and 
the fact that for all $\alpha \in [0,1], r \in [1,\infty)$ it holds that $1 \leq r^\alpha \leq r$ 
show that for all  $\alpha \in [0,1], r \in (0,\infty)$ it holds that
\begin{equation}
\abs{1-r^\alpha} \leq \abs{1-r}.
\end{equation}
Combining this and 
the fact that for all $l \in \N$ it holds that $n_l - \floor{n_l} \in [0,1]$ with 
the hypothesis that $\liminf_{ l \to \infty } n_l = \infty$ and 
the fact that $\sup_{l \in \N} \abs{a_l} < \infty$ 
demonstrates that 
\begin{equation}
	\limsup_{ l \to \infty }
		\left| 
			1 - \left[ 1 + \tfrac{a_l}{ n_l }\right]^{ n_l - \floor{n_l} }
		\right|
\leq	
	\limsup_{ l \to \infty }
		\left| 
			1 - \left[ 1 + \tfrac{a_l}{ n_l }\right]
		\right|
=
	\limsup_{ l \to \infty }
		\big| 
			\tfrac{a_l}{ n_l }
		\big|
=
	0.
\end{equation}
This and \eqref{fundamental_limit:eq2} establish that
\begin{equation}
\begin{split}
\lim_{ l \to \infty } 
	\Big[ 
		\left[ 1 + \tfrac{a_l}{ n_l }\right]^{ n_l }
	\Big]
&=
\lim_{ l \to \infty } 
	\left[ 
		\left[ 1 + \tfrac{a_l}{ n_l }\right]^{ \floor{n_l} }
		\left[ 1 + \tfrac{a_l}{ n_l }\right]^{ n_l - \floor{n_l} }
	\right] \\
&=
\lim_{ l \to \infty } 
	\left[ 
		\left[ 1 + \tfrac{ \left(\frac{a_l\floor{n_l} }{n_l}\right)}{\floor{n_l} }  \right]^{ \floor{n_l} }
		\left[ 1 + \tfrac{a_l}{ n_l }\right]^{ n_l - \floor{n_l} }
	\right] \\
&=
	\left[ 
		\lim_{ l \to \infty } \left[ 1 + \tfrac{ \left(\frac{a_l\floor{n_l} }{n_l}\right)}{\floor{n_l} }  \right]^{ \floor{n_l} }
	\right]
	\left[ 
		\lim_{ l \to \infty } \left[ 1 + \tfrac{a_l}{ n_l }\right]^{ n_l - \floor{n_l} }
	\right]\\
&=
	\exp \! \left( \lim_{ l \to \infty }  a_l\right).
\end{split}
\end{equation}
The proof of Lemma~\ref{fundamental_limit} is thus completed.
\end{proof}

\subsubsection{Lower error estimates}
\label{subsubsect:LB}

\begin{lemma}
\label{aux_LB_2}
Let $\beta \in (0,\infty)$, $\nu \in (0,1]$. Then
\begin{equation}
\liminf_{n \to \infty} 
\left[ n^\nu \!
	\left( 
		\sum_{k = 1}^n \left[ \tfrac{\beta}{k^\nu}\left( \prod_{l = k+1}^n (1-\tfrac{\beta}{l^\nu})\right) \right]^{\!2} 
	\right)
\right]
\geq
 \frac{\beta^2\exp(-2^\nu\beta )}{2}  .
\end{equation}
\end{lemma}

\begin{proof}[Proof of Lemma~\ref{aux_LB_2}]
Throughout this proof let $\ceil{\cdot} \colon \R \to \Z$ be the function which satisfies for all $x \in \R$ that $\ceil{x} = \min( [x,\infty) \cap  \Z)$. 
Observe that the fact that for all $n \in \N$ it holds that $\ceil*{n-\frac{n^\nu}{2}} \geq n-\frac{n^\nu}{2} \geq n - \frac{n}{2} = \frac{n}{2} > 0$ ensures that
\begin{equation}
\begin{split}
&\liminf_{n \to \infty} \left[ n^\nu \! \left( \sum_{k = 1}^n \left[ \tfrac{\beta}{k^\nu}\left( \prod_{l = k+1}^n (1-\tfrac{\beta}{l^\nu})\right) \right]^{\!2} \right) \right]  \\
&\geq 
\liminf_{n \to \infty}  \left[ n^\nu \! \left(  \sum_{k =\ceil*{n-\frac{n^\nu}{2}} }^n \left[ \tfrac{\beta}{k^\nu}\left( \prod_{l = k+1}^n (1-\tfrac{\beta}{l^\nu})\right) \right]^{\!2}\right) \right] \\
&\geq 
\liminf_{n \to \infty}  \left[ n^\nu \! \left( \sum_{k =\ceil*{n-\frac{n^\nu}{2}} }^n \left[ \tfrac{\beta}{n^\nu}\left( \prod_{l = \ceil*{n-\frac{n^\nu}{2}}+1}^n (1-\tfrac{\beta}{l^\nu})\right) \right]^{\!2}\right) \right] \\
&\geq 
\liminf_{n \to \infty} \left[  n^\nu   \left[\tfrac{\beta}{n^{\nu}}\right]^2 \! \left(  \sum_{k =\ceil*{n-\frac{n^\nu}{2}} }^n \left[ \left[1-\tfrac{\beta}{(\ceil*{n-\frac{n^\nu}{2}} )^\nu} \right]^{n - \ceil*{n-\frac{n^\nu}{2}} } \right]^{\!2} \right) \right] \\
&= 
\liminf_{n \to \infty} \left[ \tfrac{\beta^2}{n^{\nu}} \left(n - \ceil*{n-\tfrac{n^\nu}{2}}+ 1\right) \left[ \left[1-\tfrac{\beta}{(\ceil*{n-\frac{n^\nu}{2}} )^\nu} \right]^{n - \ceil*{n-\frac{n^\nu}{2}} } \right]^{\!2} \right] .
\end{split}
\end{equation}
The fact that for all $n \in \N$ it holds that $\ceil*{n-\frac{n^\nu}{2}} \leq n-\frac{n^\nu}{2} + 1$ and the fact that for all $n \in \N$ it holds that $\ceil*{n-\frac{n^\nu}{2}} \geq n-\frac{n^\nu}{2} \geq n - \frac{n}{2} = \frac{n}{2}$ hence demonstrate that 
\begin{equation}
\begin{split} 
&\liminf_{n \to \infty}\left[ n^\nu \! \left( \sum_{k = 1}^n \left[ \tfrac{\beta}{k^\nu}\left( \prod_{l = k+1}^n (1-\tfrac{\beta}{l^\nu})\right) \right]^{\!2}\right) \right]  \\
&\geq
\liminf_{n \to \infty} \left[\tfrac{\beta^2}{n^{\nu}} \left(n - (n-\tfrac{n^\nu}{2} + 1)+ 1\right) \left[ \left[1-\tfrac{\beta}{(\frac{n}{2} )^\nu} \right]^{n - \left(n-\frac{n^\nu}{2}\right)} \right]^{\!2} \right] \\
&= 
\liminf_{n \to \infty} \left[\tfrac{\beta^2}{n^{\nu}} \! \left( \tfrac{n^\nu}{2} \! \left[1-\tfrac{2^\nu\beta}{n^\nu} \right]^{n^\nu} \right) \right]
= 
\liminf_{n \to \infty} \left(\tfrac{\beta^2}{2} \! \left[1-\tfrac{2^\nu\beta}{n^\nu} \right]^{n^\nu} \right) \\
&=
\tfrac{\beta^2}{2} \left[ \liminf_{n \to \infty} \left( \left[1-\tfrac{2^\nu\beta}{n^\nu} \right]^{n^\nu} \right)\right].
\end{split}
\end{equation}
Combining this with Lemma~\ref{fundamental_limit} (with
$a_l =  -2^\nu\beta$,
$n_l = l^\nu$ 
for $l \in \N$ in the notation of Lemma~\ref{fundamental_limit})
establishes that
\begin{equation}
\liminf_{n \to \infty} \left[n^\nu \!\left( \sum_{k = 1}^n \left[ \tfrac{\beta}{k^\nu}\left( \prod_{l = k+1}^n (1-\tfrac{\beta}{l^\nu})\right) \right]^{\!2} \right) \right] 
\geq 
\frac{\beta^2 \exp(-2^\nu\beta)}{2}.
\end{equation}
The proof of Lemma~\ref{aux_LB_2} is thus completed.
\end{proof}

\begin{prop}
\label{LB}
Assume Setting~\ref{setting} and assume that $\nu \leq 1$ and $\EXP{\norm{X_1-\EXp{X_1}}^2} > 0$.
Then there exists $C \in (0,\infty)$ such that for all $n \in \N$ it holds that
\begin{equation}
\big(\EXP{\norm{\Theta_n-\EXp{X_1}}^2}\big)^{\nicefrac{1}{2}} 
\geq
 C n^{-\nicefrac{\nu}{2}}.
\end{equation}
\end{prop}

\begin{proof}[Proof of Proposition~\ref{LB}] 
First, observe that item~\eqref{explicit_expressions:item4} in Proposition~\ref{explicit_expressions} ensures that for all $n \in \N$ it holds that
\begin{equation}
\label{LB:eq1}
\begin{split}
\Exp{\norm{\Theta_n - \EXp{X_1}}^2} 
&=
\left[ \prod_{l = 1}^n (1-\tfrac{\gamma\alpha}{l^\nu})\right]^2 \norm{\xi-\EXp{X_1}}^2 \\
&\quad + 
\Exp{\norm{X_1 - \EXp{X_1}}^2}
\left[
\sum_{k = 1}^n 
\left[
\tfrac{\gamma\alpha}{k^\nu}\left( \prod_{l = k+1}^n (1-\tfrac{\gamma\alpha}{l^\nu})\right)
\right]^2
\right] \\
&\geq
\Exp{\norm{X_1 - \EXp{X_1}}^2}
\left[
\sum_{k = 1}^n 
\left[
\tfrac{\gamma\alpha}{k^\nu}\left( \prod_{l = k+1}^n (1-\tfrac{\gamma\alpha}{l^\nu})\right)
\right]^2
\right].
\end{split}
\end{equation}
Moreover, note that Lemma~\ref{aux_LB_2} (with $\beta = \gamma \alpha$, $\nu = \nu$ in the notation of Lemma~\ref{aux_LB_2}) implies that there exists $m \in \N$ such that for all $n \in \{m,m+1,\ldots \}$ it holds that
\begin{equation}
n^\nu \!
	\left(
	 	\sum_{k = 1}^n 
	 		\left[ 
	 			\tfrac{\gamma\alpha}{k^\nu}\left( \prod_{l = k+1}^n (1-\tfrac{\gamma\alpha}{l^\nu})\right) 
	 		\right]^{\!2} 
	 \right)
\geq 
\tfrac{1}{2} \big( \tfrac{(\gamma	\alpha)^2 \exp(-2^\nu\gamma\alpha)}{2} \big) 
= 
 \tfrac{(\gamma\alpha)^2\exp(-2^\nu\gamma\alpha)}{4}  . 
\end{equation}
Therefore, we obtain for all $n \in \{m,m+1,\ldots \}$ that
\begin{equation}
\sum_{k = 1}^n \left[ \tfrac{\gamma\alpha}{k^\nu}\left( \prod_{l = k+1}^n (1-\tfrac{\gamma\alpha}{l^\nu})\right) \right]^{\!2} 
\geq 
\left[ \tfrac{(\gamma\alpha)^2\exp(-2^\nu\gamma\alpha)}{4} \right]  n^{-\nu} . 
\end{equation}
This and (\ref{LB:eq1}) demonstrate that for all $n \in \{m,m+1,\ldots \}$ it holds that
\begin{equation}
\label{LB:eq2}
\begin{split}
\EXP{\norm{\Theta_n-\EXp{X_1}}^2} 
&\geq 
\EXP{\norm{X_{1}- \EXp{X_1}}^2}
\left[
\sum_{k = 1}^n 
\left[
\tfrac{\gamma\alpha}{k^\nu}
\left(
\prod_{l = k+1}^n (1-\tfrac{\gamma\alpha}{l^\nu})
\right)
\right]^{\!2}
\right] \\
&\geq 
\left[\EXP{\norm{X_{1}- \EXp{X_1}}^2} \! \left( \tfrac{(\gamma\alpha)^2\exp(-2^\nu\gamma\alpha)}{4} \right)\right] n^{-\nu}.
\end{split}
\end{equation}
Furthermore, observe that Lemma~\ref{positivity} and the hypothesis that $\EXP{\norm{X_1-\EXp{X_1}}^2} > 0$ prove that for all $n \in \N \cap (0,m)$ it holds that
\begin{equation}
\begin{split}
\EXP{\norm{\Theta_n-\EXp{X_1}}^2} > 0.
\end{split}
\end{equation}
Hence, we obtain for all $n \in \N \cap (0,m)$ that
\begin{equation}
\begin{split}
\EXP{\norm{\Theta_n-\EXp{X_1}}^2} 
&= 
\left[ \tfrac{\EXp{\norm{\Theta_n-\EXp{X_1}}^2}}{n^{-\nu}}\right] n^{-\nu} \\
&\geq 
\left[\min \! \left\{ \tfrac{\EXp{\norm{\Theta_k- \EXp{X_1}}^2}}{k^{-\nu}} \colon k  \in \N \cap (0,m) \right\} \right] n^{-\nu} > 0.
\end{split}
\end{equation}
Combining this, \eqref{LB:eq2}, and the hypothesis that $\EXP{\norm{X_1-\EXp{X_1}}^2} > 0$ assures that for all $n \in \N$ it holds that
\begin{equation}
\begin{split}
	&\EXP{\norm{\Theta_n-\EXp{X_1}}^2}\\
&\geq
	\bigg[ \! \min \! 
		\bigg( \! 
			\left\{ 
				\tfrac{\EXp{\norm{\Theta_k- \EXp{X_1}}^2}}{k^{-\nu}} \colon k \in \N \cap (0,m)
			\right\} \\
			&\quad \cup 
			\left\{  
				\left[\tfrac{(\gamma\alpha)^2\exp(-2^\nu\gamma\alpha)}{4} \right] \EXP{\norm{X_{1}- \EXp{X_1}}^2} 
			\right\} \! 
		\bigg) 
	\bigg]
	 n^{-\nu}
> 0.
\end{split}
\end{equation}
Therefore, we obtain for all $n \in \N$ that 
\begin{equation}
\begin{split}
	&\big(\EXP{\norm{\Theta_n-\EXp{X_1}}^2}\big)^{\nicefrac{1}{2}} \\
&\geq
	n^{-\nicefrac{\nu}{2}}
	\Bigg[ 
		\min \! 
		\bigg( \!
		 	\left\{ 
		 		\tfrac{\EXp{\norm{\Theta_k- \EXp{X_1}}^2}}{k^{-\nu}} \colon  k \in \N \cap (0,m)
		 	\right\} \\
			&\quad \cup 
			\left\{ 
				\left[\tfrac{(\gamma\alpha)^2\exp(-2^\nu\gamma\alpha)}{4} \right] \EXP{\norm{X_{1}- \EXp{X_1}}^2} 
			\right\} \!
		\bigg)
	\Bigg]^{\nicefrac{1}{2}} 
> 
0.
\end{split}
\end{equation}
The proof of Proposition~\ref{LB} is thus completed.
\end{proof}

\subsection{Refined lower errors estimates in the case of fast decaying learning rates}
\label{subsect:LB_fast_refined}
In this subsection we establish in Lemma~\ref{LB_nueq1_deterministic} below a lower error bound for the SGD process in \eqref{setting:eq2} in the case of fast decaying learning rates (corresponding to the case $\nu = 1$ in Setting~\ref{setting}). 
Combining this lower bound with the lower bound from Proposition~\ref{LB} (see Lemma~\ref{LB_nueq1_variance} below) allows us to establish the refined lower error bound in Proposition~\ref{LB_nueq1} below.

\subsubsection{An estimate for the natural logarithm}
\label{subsubsect:LB_log}

In Lemma~\ref{LB_log} below we recall an elementary and well-known property of the natural logarithm (see, e.g., \cite{WikipediaLBlog}). 
Lemma~\ref{LB_log} will be employed in our proof of Lemma~\ref{LB_nueq1_aux} which, in turn, will be used to prove Lemma~\ref{LB_nueq1_deterministic}.
For completeness we provide the proof of Lemma~\ref{LB_log} here.

\begin{lemma}
\label{LB_log}
It holds for all $x \in (0,\infty)$ that
\begin{equation}
\log(x) \geq \frac{(x-1)}{x}.
\end{equation}
\end{lemma}

\begin{proof}[Proof of Lemma~\ref{LB_log}]
Throughout this proof let $f \colon (0,\infty) \to \R$ be the function which satisfies for all $x \in (0,\infty)$ that
\begin{equation}
f(x) = \log(x) - \frac{(x-1)}{x} = \log(x) - 1 + \frac{1}{x} = \log(x) - 1 + x^{-1} .
\end{equation}
Note that 
\begin{equation}
\label{LB_log:eq1}
f(1) = \log(1) - \frac{(1-1)}{1} = \log(1) = 0.
\end{equation}
Moreover, observe that for all $x \in (0,\infty)$ it holds that
\begin{equation}
\label{LB_log:eq2}
f'(x) = \frac{1}{x} - \frac{1}{x^2} =\frac{(x - 1)}{x^2} .
\end{equation}
This ensures that for all $x \in [1,\infty)$ it holds that $f'(x) \geq 0$.
The fundamental theorem of calculus and (\ref{LB_log:eq1}) hence imply that for all $x \in [1,\infty)$ it holds that
\begin{equation}
\label{LB_log:eq3}
f(x) = f(1) + \int_1^x f'(t) \, \mathrm{d}t \geq f(1) = 0.
\end{equation}
Moreover, note that (\ref{LB_log:eq2}) assures that for all $x \in (0,1]$ it holds that $f'(x) \leq 0$.
The fundamental theorem of calculus and (\ref{LB_log:eq1}) therefore ensure that for all $x \in (0,1]$ it holds that
\begin{equation}
0 = f(1) = f(x) + \int_x^1 f'(t) \, \mathrm{d}t \leq f(x).
\end{equation}
Combining this with (\ref{LB_log:eq3}) proves that for all $x \in (0,\infty)$ it holds that
\begin{equation}
\log(x) - \frac{(x-1)}{x} = f(x) \geq 0.
\end{equation}
Therefore, we obtain for all $x \in (0,\infty)$ that
\begin{equation}
\log(x) \geq \frac{(x-1)}{x}.
\end{equation}
The proof of Lemma~\ref{LB_log} is thus completed.
\end{proof}

\subsubsection{Errors due to the deterministic gradient descent dynamic}
\label{subsubsect:LB_deterministic}

\begin{lemma}
\label{LB_nueq1_aux}
Let $m \in \N$, $\beta \in (0,\infty) \backslash \{m,m+1,m+2,\ldots \}$. 
Then for every $\varepsilon \in (0,\infty)$ there exists $C \in (0,\infty)$ such that for all $n \in \N \cap [m,\infty)$ it holds that
\begin{equation}
\label{LB_nueq1_aux:concl1}
 \left[ \prod_{l = m}^n \Abs{1-\tfrac{\beta}{l}} \right]
 \geq
 C n^{-(\beta + \varepsilon)}. 
\end{equation}
\end{lemma}

\begin{proof}[Proof of Lemma~\ref{LB_nueq1_aux}]
Throughout this proof let $\varepsilon \in (0,\infty)$, let $L \in \N \cap (\max\{m, \beta \},\infty)$ satisfy that
\begin{equation}
\label{LB_nueq1_aux:eq1}
\frac{\beta}{(1-\tfrac{\beta }{L})} \leq (\beta + \varepsilon),
\end{equation}
and let $C \in [0,\infty)$ be given by
\begin{equation}
\label{LB_nueq1_aux:eq2}
C 
= 
\min \! 
\left(
	\left\{ 
		\tfrac{\left[ \prod_{l = m}^k \abs{1-\frac{\beta}{l}} \right]}{k^{-(\beta + \varepsilon)}} \colon k \in \N \cap [m,L]
	\right\} \cup 
	\left\{ 
		{\textstyle \prod\limits_{l = m}^L \Abs{1-\tfrac{\beta}{l}} }
	\right\}
\right).
\end{equation}
Note that the fact that $\beta \notin \{m,m+1,m+2, \ldots \} = \N \cap [m,\infty)$ ensures that for all $ l \in \N \cap [m,\infty)$ it holds that
\begin{equation}
\Abs{1-\tfrac{\beta }{l}} > 0.
\end{equation}
This and \eqref{LB_nueq1_aux:eq2} establish that $C > 0$.
Moreover, observe that the fact that $L > \beta$ assures that for all $ l \in \N \cap [L,\infty)$ it holds that
\begin{equation}
0 = 1 - 1 < 1 - \tfrac{\beta}{L} \leq 1 - \tfrac{\beta}{l} < 1-0 = 1.
\end{equation}
Hence, we obtain that for all $l \in \N \cap [L,\infty)$ it holds that
\begin{equation}
\label{LB_nueq1_aux:eq3}
\left(1-\tfrac{\beta }{l}\right) \in (0,1).
\end{equation}
Lemma~\ref{LB_log} and \eqref{LB_nueq1_aux:eq1} therefore assure that for all $n \in \N \cap (L,\infty)$ it holds that
\begin{equation}
\begin{split}
&\log \! \left(  \prod_{l = L+1}^n \Abs{1-\tfrac{\beta}{l}}\right)
=
\sum_{l = L+1}^n  \log \! \left(   1-\tfrac{\beta}{l} \right) \\
&\geq
 \sum_{l = L+1}^n \left( \frac{\left(1-\tfrac{\beta}{l} \right) - 1}{\left(1-\tfrac{\beta}{l}\right)} \right) 
=
- \left[ \sum_{l = L+1}^n  \left(\frac{1}{l} \left[\frac{ \beta}{\left(1-\tfrac{\beta}{l}\right)} \right] \right) \right]  \\
&\geq
- \left[ \sum_{l = L+1}^n  \left(\frac{1}{l} \left[\frac{ \beta}{\left(1-\tfrac{\beta}{L}\right)} \right] \right) \right]  
\geq
- \left[\sum_{l = L+1}^n  \frac{(\beta + \varepsilon)}{l} \right]\\
&\geq 
- (\beta + \varepsilon) \left[\sum_{l = 2}^n  \frac{1}{l} \right].
\end{split}
\end{equation}
The fact that for all $n \in \N$ it holds that
\begin{equation}
\sum_{l = 2}^n  \frac{1}{l} 
= 
\sum_{l = 2}^n  \left[ \int_{l-1}^l \frac{1}{l} \, \mathrm{d}x \right]
\leq  
\sum_{l = 2}^n \left[ \int_{l-1}^l \frac{1}{x} \, \mathrm{d}x \right]
= 
\int_{1}^n \frac{1}{x} \, \mathrm{d}x
= 
\log(n)
\end{equation}
hence ensures that for all $n \in \N \cap (L,\infty)$ it holds that
\begin{equation}
\begin{split}
 \prod_{l = L+1}^n \Abs{1-\tfrac{\beta}{l}}
 &= 
\exp \! \left( \log \! \left( \prod_{l = L+1}^n \Abs{1-\tfrac{\beta}{l}} \right)\right) \\
&\geq
\exp \! \left( -(\beta + \varepsilon) \left[ \sum_{l = 2}^n  \frac{1}{l} \right] \right) \\
&\geq 
\exp \! \big( - (\beta + \varepsilon)\log(n) \big) 
= 
n^{- (\beta + \varepsilon)}.
\end{split}
\end{equation}
This and \eqref{LB_nueq1_aux:eq2} demonstrate that for all $n \in \N \cap (L,\infty)$ it holds that
\begin{equation}
\label{LB_nueq1_aux:eq4}
\begin{split}
	\prod_{l = m}^n \Abs{1-\tfrac{\beta}{l}}
&=
	\left[\prod_{l = m}^L \Abs{1-\tfrac{\beta}{l}}\right] \left[ \prod_{l = L+1}^n \Abs{1-\tfrac{\beta}{l}} \right] \\
&\geq 
	\left[\prod_{l = m}^L \Abs{1-\tfrac{\beta}{l}}\right] n^{- (\beta + \varepsilon)} 
\geq 
	C n^{- (\beta + \varepsilon)} .
\end{split}
\end{equation}
Moreover, note that \eqref{LB_nueq1_aux:eq2} implies that for all $n \in \N \cap [m,L]$ it holds that
\begin{equation}
\begin{split}
 \prod_{l = m}^n \Abs{1-\tfrac{\beta}{l}}
 =
 \left[
 	\tfrac{\left[ \prod_{l = m}^n \abs{1-\frac{\beta}{l}} \right]}{n^{-(\beta + \varepsilon)}} 
 \right]
 n^{-(\beta + \varepsilon)}
 \geq
 C n^{-(\beta + \varepsilon)}.
\end{split}
\end{equation} 
Combining this and \eqref{LB_nueq1_aux:eq4} establishes that for all $n \in \N \cap [m,\infty)$ it holds that
\begin{equation}
\begin{split}
 \prod_{l = m}^n \Abs{1-\tfrac{\beta}{l}}
 \geq
 C n^{- (\beta + \varepsilon)}.
\end{split}
\end{equation}
The fact that $C > 0$ therefore establishes \eqref{LB_nueq1_aux:concl1}.
The proof of Lemma~\ref{LB_nueq1_aux} is thus completed.
\end{proof}


\begin{lemma}[Lower bound for deterministic gradient descent]
\label{LB_deterministic}
Let $d \in \N$, $\kappa \in \R$, $\vartheta \in \R^d$, $\xi \in \R^d \backslash \{\vartheta \}$, $\alpha \in (0,\infty)$, $\gamma \in (0,\infty) \backslash \{\frac{1}{\alpha}, \frac{2}{\alpha}, \frac{3}{\alpha}, \ldots \}$, 
let $\lrnorm{\cdot} \colon \R^d \to [0,\infty)$ be the $d$-dimensional Euclidean norm, 
let $f \colon \R^d \to \R$ be the function which satisfies for all $\theta  \in \R^d$ that
\begin{equation}
\label{LB_deterministic:ass1}
f(\theta) = \tfrac{\alpha}{2} \norm{\theta - \vartheta}^2 + \kappa,
\end{equation}
and let $\Theta \colon \N_0 \times \Omega \to \R^d$ be the function which satisfies for all $n \in \N$ that 
\begin{equation}
\Theta_0 = \xi \qandq \Theta_n = \Theta_{n-1} - \tfrac{\gamma}{n} (\nabla f) (\Theta_{n-1}).
\end{equation}
Then
\begin{enumerate}[(i)]
\item \label{LB_deterministic:item1}
it holds that
$
\{\theta \in \R^d  \colon  f(\theta) = \inf\nolimits_{w \in \R^d} f(w)  \} = \{ \vartheta \}
$
and

\item \label{LB_deterministic:item2}
for every $\varepsilon \in (0,\infty)$ there exists $C \in (0,\infty)$ such that for all $n \in \N$ it holds that
\begin{equation}
\norm{\Theta_n -\vartheta} \geq C n^{-(\gamma \alpha + \varepsilon)}.
\end{equation}
\end{enumerate}
\end{lemma}

\begin{proof}[Proof of Lemma~\ref{LB_deterministic}]
Throughout this proof let $\varepsilon \in (0,\infty)$.
Observe that \eqref{LB_deterministic:ass1} proves item~\eqref{LB_deterministic:item1}.
It thus remains to prove item~\eqref{LB_deterministic:item2}.
For this note that Lemma~\ref{der_of_norm} and \eqref{LB_deterministic:ass1} ensure that for all $\theta \in \R^d$ it holds that
\begin{equation}
(\nabla f)(\theta) 
= 
\tfrac{\alpha}{2}( 2(\theta - \vartheta) )
=
\alpha(\theta - \vartheta).
\end{equation}
Therefore, we obtain for all $n \in \N$ that
\begin{equation}
\begin{split}
\Theta_n -\vartheta
&= 
\Theta_{n-1} - \tfrac{\gamma}{n} (\nabla f) (\Theta_{n-1}) - \vartheta \\
&=
\Theta_{n-1} -\vartheta - \tfrac{\gamma\alpha}{n}  (\Theta_{n-1} - \vartheta)\\
&=
(1- \tfrac{\gamma\alpha}{n})(\Theta_{n-1} -\vartheta).
\end{split}
\end{equation}
Induction hence proves that for all $n \in \N$ it holds that
\begin{equation}
\Theta_n -\vartheta 
= 
\left[\prod_{l = 1}^n (1- \tfrac{\gamma\alpha}{l})\right](\Theta_{0} -\vartheta) 
= 
\left[\prod_{l = 1}^n (1- \tfrac{\gamma\alpha}{l})\right](\xi -\vartheta).
\end{equation}
This assures that for all $n \in \N$ it holds that
\begin{equation}
\label{LB_deterministic:eq1}
\norm{\Theta_n -\vartheta }
=
\left[\prod_{l = 1}^n \Abs{1- \tfrac{\gamma\alpha}{l}}\right] \norm{\xi -\vartheta}.
\end{equation}
Next observe that Lemma~\ref{LB_nueq1_aux}
(with
$m = 1$,
$\beta = \gamma \alpha$
in the notation of Lemma~\ref{LB_nueq1_aux})
and the fact that $\gamma \alpha \notin \N$ imply that there exists $C \in (0,\infty)$ such that for all $n \in \N$ it holds that
\begin{equation}
\left[\prod_{l = 1}^n \Abs{1- \tfrac{\gamma\alpha}{l}}\right]  
\geq
C n^{- (\gamma\alpha + \varepsilon)}.
\end{equation}
Combining this with \eqref{LB_deterministic:eq1} demonstrates that for all $n \in \N$ it holds that
\begin{equation}
	\norm{\Theta_n -\vartheta }
\geq
	\big[ C  n^{- (\gamma\alpha + \varepsilon)} \big] \norm{\xi -\vartheta} 
=
	\big[ C \norm{\xi -\vartheta} \big] n^{- (\gamma\alpha + \varepsilon)} .
\end{equation}
The hypothesis that $\xi \neq \vartheta$ hence establishes item~\eqref{LB_deterministic:item2}.
The proof of Lemma~\ref{LB_deterministic} is thus completed.
\end{proof}

\begin{lemma}
\label{LB_nueq1_deterministic}
Assume Setting~\ref{setting} and assume that $\nu =1$ and $\EXp{\norm{X_1-\EXp{X_1}}^2}  > 0$.
Then for every $\varepsilon \in (0,\infty)$ there exists $C \in (0,\infty)$ such that for all $n \in \N$ it holds that
\begin{equation}
\label{LB_nueq1_deterministic:concl1}
\big(\EXP{\norm{\Theta_n-\EXp{X_1}}^2}\big)^{\nicefrac{1}{2}} 
\geq
 C n^{-(\gamma \alpha + \varepsilon)}.
\end{equation}
\end{lemma}

\begin{proof}[Proof of Lemma~\ref{LB_nueq1_deterministic}]
Throughout this proof 
let $\varepsilon \in (0,\infty)$, let $m \in \N \cap (\gamma\alpha -1 ,\infty)$, and let $\mathcal{M} \in [0,\infty)$ be given by
\begin{equation}
\mathcal{M} 
=
\min \! \left\{ \tfrac{\left(\EXp{\norm{\Theta_k- \EXp{X_1}}^2}\right)^{ \! \nicefrac{1}{2}}}{k^{- (\gamma\alpha + \varepsilon)}} \colon k \in \{1,2,\ldots,m \}\right\}.
\end{equation}
Observe that Lemma~\ref{positivity} and the hypothesis that $\EXP{\norm{X_1-\EXp{X_1}}^2}>0$ assure that for all $n \in \{1,2,\ldots,m \}$ it holds that
\begin{equation}
\begin{split}
\EXP{\norm{\Theta_n-\EXp{X_1}}^2} > 0.
\end{split}
\end{equation}
This ensures that $\mathcal{M}  > 0$.
Next note that item~\eqref{explicit_expressions:item4} in Proposition~\ref{explicit_expressions} assures that for all $n \in \N \cap (m,\infty)$ it holds that
\begin{equation}
\label{LB_nueq1_deterministic:eq1}
\begin{split}
&\Exp{\norm{\Theta_n - \EXp{X_1}}^2} \\
&=
\left[ \prod_{l = 1}^n (1-\tfrac{\gamma\alpha}{l})\right]^2 \! \norm{\xi-\EXp{X_1}}^2  \\
&\quad +
\Exp{\norm{X_1 - \EXp{X_1}}^2}
\left[
\sum_{k = 1}^n 
\left[
\tfrac{\gamma\alpha}{k}\left( \prod_{l = k+1}^n (1-\tfrac{\gamma\alpha}{l})\right)
\right]^2
\right] \\
&\geq
\Exp{\norm{X_1 - \EXp{X_1}}^2} 
\left[
	\tfrac{\gamma\alpha}{m}\left( \prod_{l = m+1}^n (1-\tfrac{\gamma\alpha}{l})\right)
\right]^2 \\
&=
(\tfrac{\gamma\alpha}{m})^2 \,
\Exp{\norm{X_1 - \EXp{X_1}}^2} 
\left[
	\prod_{l = m+1}^n \Abs{1-\tfrac{\gamma\alpha}{l}}
\right]^2.
\end{split}
\end{equation}
Moreover, observe that the fact that $m +1> \gamma\alpha$ ensures that $\gamma \alpha \notin \{m+1, m+2, \ldots\}$.
Lemma~\ref{LB_nueq1_aux} 
(with 
$m = m+1$, 
$\beta = \gamma\alpha$ 
in the notation of Lemma~\ref{LB_nueq1_aux})
therefore demonstrates that there exists $C \in (0,\infty)$ such that for all $n \in \N \cap [m+1,\infty) = \N \cap (m,\infty)$ it holds that
\begin{equation}
 \left[ \prod_{l = m+1}^n \Abs{1-\tfrac{\gamma\alpha}{l}} \right]
 \geq
 C n^{-(\gamma\alpha + \varepsilon)}. 
\end{equation}
Combining this with \eqref{LB_nueq1_deterministic:eq1} proves that for all $n \in \N \cap (m,\infty)$ it holds that
\begin{equation}
\label{LB_nueq1_deterministic:eq2}
\begin{split}
\big(\EXP{\norm{\Theta_n-\EXp{X_1}}^2}\big)^{\nicefrac{1}{2}}  
&\geq
\left[\tfrac{\gamma\alpha\left(\EXp{\norm{X_1 - \EXp{X_1}}^2} \right)^{\!\nicefrac{1}{2}} }{m} \right]
\left[
	\prod_{l = m+1}^n \Abs{1-\tfrac{\gamma\alpha}{l}}
\right] \\
&\geq
\left[\tfrac{\gamma\alpha C\left(\EXp{\norm{X_1 - \EXp{X_1}}^2} \right)^{\!\nicefrac{1}{2}} }{m} \right]
n^{-(\gamma\alpha + \varepsilon)}.
\end{split}
\end{equation}
In addition, note that for all $n \in \{1,2,\ldots,m\}$ it holds that
\begin{equation}
\label{LB_nueq1_deterministic:eq3}
\begin{split}
\big(\EXP{\norm{\Theta_n-\EXp{X_1}}^2}\big)^{\nicefrac{1}{2}}  
=
\left[\tfrac{\left(\EXp{\norm{\Theta_n - \EXp{X_1}}^2} \right)^{\!\nicefrac{1}{2}}}{n^{-(\gamma\alpha + \varepsilon)}} \right] n^{-(\gamma\alpha + \varepsilon)} 
\geq
\mathcal{M} n^{-(\gamma\alpha + \varepsilon)}.
\end{split}
\end{equation}
This and \eqref{LB_nueq1_deterministic:eq2} establish that for all $n \in \N$ it holds that
\begin{equation}
\begin{split}
	\big(\EXP{\norm{\Theta_n-\EXp{X_1}}^2}\big)^{\nicefrac{1}{2}}  
\geq
	\left[
		\min \! \left\{ \mathcal{M} , \left[\tfrac{\gamma\alpha C\left(\EXp{\norm{X_1 - \EXp{X_1}}^2} \right)^{\!\nicefrac{1}{2}} }{m} \right] \right\} 
	\right] n^{-(\gamma\alpha + \varepsilon)}.
\end{split}
\end{equation}
The hypothesis that $\EXp{\norm{X_1 - \EXp{X_1}}^2} > 0$ and the fact that $\mathcal{M} > 0$ therefore establish \eqref{LB_nueq1_deterministic:concl1}.
The proof of Lemma~\ref{LB_nueq1_deterministic} is thus completed.
\end{proof}

\subsubsection{Errors due to the randomness in the SGD method}
\label{subsubsect:LB_stochastic}

\begin{lemma}
\label{LB_nueq1_variance}
Assume Setting~\ref{setting} and assume that $\nu =1$ and $\EXp{\norm{X_1-\EXp{X_1}}^2}  > 0$.
Then there exists $C \in (0,\infty)$ such that for all $n \in \N$ it holds that
\begin{equation}
\label{LB_nueq1_variance:concl1}
\big(\EXP{\norm{\Theta_n-\EXp{X_1}}^2}\big)^{\nicefrac{1}{2}} 
\geq
 C n^{-\nicefrac{1}{2} }.
\end{equation}
\end{lemma}

\begin{proof}[Proof of Lemma~\ref{LB_nueq1_variance}]
Note that Proposition~\ref{LB} and the hypothesis that $\nu = 1$ demonstrate that there exists $C \in (0,\infty)$ such that for all $n \in \N$ it holds that
\begin{equation}
\big(\EXP{\norm{\Theta_n-\EXp{X_1}}^2}\big)^{\nicefrac{1}{2}} 
\geq
 C n^{-\nicefrac{\nu}{2} } 
 =  
 C n^{-\nicefrac{1}{2} }.
\end{equation}
The proof of Lemma~\ref{LB_nueq1_variance} is thus completed.
\end{proof}

\subsubsection{Composition of the errors}
\label{subsubsect:LB_fast_composition}

\begin{prop}
\label{LB_nueq1}
Assume Setting~\ref{setting} and assume that $\nu =1$ and $\EXp{\norm{X_1-\EXp{X_1}}^2}  > 0$.
Then for every $\varepsilon \in (0,\infty)$ there exists $C \in (0,\infty)$ such that for all $n \in \N$ it holds that
\begin{equation}
\label{LB_nueq1:concl1}
\big(\EXP{\norm{\Theta_n-\EXp{X_1}}^2}\big)^{\nicefrac{1}{2}} 
\geq
 C n^{-\min \{ \nicefrac{1}{2},\gamma \alpha + \varepsilon \} }.
\end{equation}
\end{prop}

\begin{proof}[Proof of Proposition~\ref{LB_nueq1}]
Throughout this proof let $\varepsilon \in (0,\infty)$.
Note that Lemma~\ref{LB_nueq1_variance} demonstrates that there exists $c \in (0,\infty)$ such that for all $n \in \N$ it holds that
\begin{equation}
\label{LB_nueq1:eq1}
\begin{split}
\big(\EXP{\norm{\Theta_n-\EXp{X_1}}^2}\big)^{\nicefrac{1}{2}} 
\geq
c n^{-\nicefrac{1}{2}}.
\end{split}
\end{equation}
Moreover, observe that Lemma~\ref{LB_nueq1_deterministic} assures that there exists $C \in (0,\infty)$ such that for all $n \in \N$ it holds that
\begin{equation}
\begin{split}
\big(\EXP{\norm{\Theta_n-\EXp{X_1}}^2}\big)^{\nicefrac{1}{2}} 
\geq
C n^{-(\gamma\alpha + \varepsilon)}.
\end{split}
\end{equation}
Combining this and \eqref{LB_nueq1:eq1} ensures that for all $n \in \N$ it holds that
\begin{equation}
\begin{split}
	\big(\EXP{\norm{\Theta_n-\EXp{X_1}}^2}\big)^{\nicefrac{1}{2}} 
&\geq
	\max \! \left\{ c n^{-\nicefrac{1}{2}} , Cn^{-(\gamma\alpha + \varepsilon)}\right\} \\
&\geq
	\min \! \left\{ c, C \right\}\max \! \left\{ n^{-\nicefrac{1}{2}}, n^{-(\gamma\alpha + \varepsilon)} \right\} \\
&=
	\big[ \! \min \! \left\{ c, C \right\} \!\big] n^{\max \{ -\nicefrac{1}{2}, -(\gamma\alpha + \varepsilon)\}} \\
&=
	\big[ \! \min \! \left\{ c, C \right\} \!\big] n^{-\min \{\nicefrac{1}{2}, \gamma\alpha + \varepsilon\}}.
\end{split}
\end{equation}
The proof of Proposition~\ref{LB_nueq1} is thus completed.
\end{proof}

\subsection{Lower errors estimates in the case of very fast decaying learning rates}
\label{subsect:LB_very_fast}

In this subsection we establish in Lemma~\ref{LB_nubigger1} below that the SGD process in \eqref{setting:eq2} fails to converge to the global minimum of the objective function in the case of very fast decaying learning rates (corresponding to the case $\nu > 1$ in Setting~\ref{setting}). 

\begin{lemma}
\label{LB_nubigger1}
Assume Setting~\ref{setting} and assume that $\nu  > 1$ and $\EXp{\norm{X_1-\EXp{X_1}}^2} > 0$.
Then there exists $C \in (0,\infty)$ such that for all $n \in \N$ it holds that
\begin{equation}
	\big(\EXP{\norm{\Theta_n-\EXp{X_1}}^2}\big)^{\nicefrac{1}{2}} 
\geq
	C.
\end{equation}
\end{lemma}

\begin{proof}[Proof of Lemma~\ref{LB_nubigger1}]
Throughout this proof let $m \in \N \cap (\gamma\alpha,\infty)$.
Observe that Lemma~\ref{positivity} and the hypothesis that $\EXp{\norm{X_1-\EXp{X_1}}^2} > 0$ ensure that for all $n \in \N$ it holds that
\begin{equation}
\label{LB_nubigger1:eq0}
\EXP{\norm{\Theta_n-\EXp{X_1}}^2} > 0.
\end{equation}
Next note that item~\eqref{explicit_expressions:item4} in Proposition~\ref{explicit_expressions} demonstrates that for all $n \in \N \cap (m,\infty)$ it holds that
\begin{equation}
\label{LB_nubigger1:eq1}
\begin{split}
\Exp{\norm{\Theta_n - \EXp{X_1}}^2} 
&=
\left[ 
	\prod_{l = 1}^n (1-\tfrac{\gamma\alpha}{l^\nu})
\right]^2 
\norm{\xi-\EXp{X_1}}^2 \\
&\quad + 
\Exp{\norm{X_1 - \EXp{X_1}}^2}
\left[
	\sum_{k = 1}^n 
		\left[
			\tfrac{\gamma\alpha}{k^\nu}\left( \prod_{l = k+1}^n (1-\tfrac{\gamma\alpha}{l^\nu})\right)
		\right]^2
\right] \\
&\geq
\Exp{\norm{X_1 - \EXp{X_1}}^2}
\left[
	\tfrac{\gamma\alpha}{m^\nu}\left( \prod_{l = m+1}^n (1-\tfrac{\gamma\alpha}{l^\nu})\right)
\right]^2 \\
&=
\left[ \tfrac{(\gamma\alpha)^2\EXp{\norm{X_1 - \EXp{X_1}}^2}  }{m^{2\nu}} \right]
\left[
	\prod_{l = m+1}^n \Abs{1-\tfrac{\gamma\alpha}{l^\nu}}
\right]^2.
\end{split}
\end{equation}
Moreover, note that the fact that $m > \gamma \alpha$ ensures that for all $l \in \N\cap[m,\infty)$ it holds that
\begin{equation}
	0 
< 
	1- \tfrac{\gamma\alpha}{m}
\leq	
	1- \tfrac{\gamma\alpha}{l}
\leq	
	1- \tfrac{\gamma\alpha}{l^\nu}
<
1.
\end{equation}
Therefore, we obtain that for all $l \in \N\cap[m,\infty)$ it holds that 
\begin{equation}
\label{LB_nubigger1:eq2}
\left(1-\tfrac{\gamma\alpha}{l^\nu}\right) \in (0,1).
\end{equation}
Lemma~\ref{LB_log} hence assures that for all $n \in \N \cap (m,\infty)$ it holds that
\begin{equation}
\label{LB_nubigger1:eq3}
\begin{split}
	&\log \! \left(  \prod_{l = m+1}^n \Abs{1-\tfrac{\gamma\alpha}{l^\nu}}\right)
=
	\sum_{l = m+1}^n  \log \! \left(   1-\tfrac{\gamma\alpha}{l^\nu} \right) \\
&\geq
	\sum_{l = m+1}^n \left( \frac{\left(1-\tfrac{\gamma\alpha}{l^\nu} \right) - 1}{\left(1-\tfrac{\gamma\alpha}{l^\nu}\right)} \right) 
=
	- \left[ \sum_{l = m+1}^n  \left( \frac{ \gamma\alpha}{l^\nu\left(1-\tfrac{\gamma\alpha}{l^\nu}\right)} \right) \right]  \\
&\geq
	- \left[\frac{ \gamma\alpha}{\left(1-\tfrac{\gamma\alpha}{m^\nu}\right)} \right] \left[ \sum_{l = m+1}^n  \frac{1}{l^\nu}  \right]  
\geq
	-\left[\frac{ \gamma\alpha}{\left(1-\tfrac{\gamma\alpha}{m^\nu}\right)} \right] 
	\left[ \sum_{l = 2}^\infty  \frac{1}{l^\nu}  \right].
\end{split}
\end{equation}
In addition, note that the hypothesis that $\nu > 1$ implies that for all $n \in \N$ it holds that
\begin{equation}
\begin{split}
	\sum_{l = 2}^\infty  \frac{1}{l^\nu} 
&= 
	\sum_{l = 2}^\infty  \left[ \int_{l-1}^l \frac{1}{l^\nu} \, \mathrm{d}x \right]
\leq  
	\sum_{l = 2}^\infty \left[ \int_{l-1}^l \frac{1}{x^\nu} \, \mathrm{d}x \right] \\
&= 
	\int_{1}^\infty x^{-\nu} \, \mathrm{d}x
= 
	\left[ \left(\tfrac{1}{(1-\nu )} \right) x^{1-\nu}\right]_{x = 1}^{x = \infty}
=
	-\tfrac{1}{(1-\nu )}
=
	\tfrac{1}{(\nu-1 )}.
\end{split}
\end{equation}
This and \eqref{LB_nubigger1:eq3} prove that for all $n \in \N \cap (m,\infty)$ it holds that
\begin{equation}
\begin{split}
	&\log \! \left(  \prod_{l = m+1}^n \Abs{1-\tfrac{\gamma\alpha}{l^\nu}}\right)
\geq
	-\left[\frac{ \gamma\alpha}{\left(1-\frac{\gamma\alpha}{m^\nu}\right)} \right]  \frac{1}{(\nu-1 )}
=
	\frac{ -\gamma\alpha}{\left(1-\frac{\gamma\alpha}{m^\nu}\right)(\nu-1 )}.
\end{split}
\end{equation}
Combining this and \eqref{LB_nubigger1:eq2}  with \eqref{LB_nubigger1:eq1} demonstrates that for all $n \in \N \cap (m,\infty)$ it holds that
\begin{equation}
\begin{split}
	\Exp{\norm{\Theta_n - \EXp{X_1}}^2} 
&\geq
	\left[ \tfrac{ (\gamma\alpha)^2\EXp{\norm{X_1 - \EXp{X_1}}^2} }{m^{2\nu}} \right]
	\exp \! \left( \! \log \! \left(
		\left[
			\prod_{l = m+1}^n \Abs{1-\tfrac{\gamma\alpha}{l^\nu}} 
		\right]^2
		\right) \right) \\
&=
	\left[ \tfrac{(\gamma\alpha)^2\EXp{\norm{X_1 - \EXp{X_1}}^2}  }{m^{2\nu}} \right]
	\exp \! \left( \! 2 \log \! \left(
			\prod_{l = m+1}^n \Abs{1-\tfrac{\gamma\alpha}{l^\nu}} 
		\right) \right) \\
&\geq
	\left[ \tfrac{(\gamma\alpha)^2\EXp{\norm{X_1 - \EXp{X_1}}^2}  }{m^{2\nu}} \right]
	\exp \! 
	\left( \!  
		\tfrac{ -2\gamma\alpha}{\left(1-\frac{\gamma\alpha}{m^\nu}\right)(\nu-1 )}
	\right).
\end{split}
\end{equation}
The hypothesis that $\EXp{\norm{X_1 - \EXp{X_1}}^2} > 0$, \eqref{LB_nubigger1:eq0}, and \eqref{LB_nubigger1:eq2} hence establish that for all $n \in \N$ it holds that
\begin{multline}
	\Exp{\norm{\Theta_n - \EXp{X_1}}^2} 
\geq
	\min \!
	\Bigg( \!
		\Big\{
			\Exp{\norm{\Theta_k - \EXp{X_1}}^2} \colon k \in \{1,2,\ldots, m\}
		\Big\} \\
		\quad \cup
		\left\{
			\left[ \tfrac{(\gamma\alpha)^2\EXp{\norm{X_1 - \EXp{X_1}}^2}  }{m^{2\nu}} \right]
			\exp \! 
			\left( \!  
				\tfrac{ -2\gamma\alpha}{\left(1-\frac{\gamma\alpha}{m^\nu}\right)(\nu-1 )}
			\right)
		\right\}	
	\Bigg)
>
	0.
\end{multline}
The proof of Lemma~\ref{LB_nubigger1} is thus completed.
\end{proof}

\subsection{Main result of this article}
\label{subsect:main_result}
The following theorem summarizes the main findings of this article.

\begin{theorem}
\label{main_theorem}
Let $d \in \N$, $\alpha,\gamma,\nu  \in (0,\infty)$, $\xi \in \R^d$, 
let $\lll \cdot, \cdot \rrr \colon \R^d \times \R^d \to \R$ be the $d$-dimensional Euclidean scalar product, 
let $\lrnorm{\cdot} \colon \R^d \to [0,\infty)$ be the $d$-dimensional Euclidean norm, 
let $(\Omega, \mathcal{F}, \P)$ be a probability space, 
let $X_n \colon \Omega \to \R^d$, $n \in \N$, be i.i.d.\ random variables with
$\Exp{\norm{X_1}^2} < \infty$
and
$\P(X_1= \EXp{X_1}) < 1$,
let $(r_{\varepsilon, i})_{\varepsilon \in (0,\infty), i \in \{0,1\}} \subseteq \R$ satisfy for all $\varepsilon \in (0,\infty)$, $i \in \{0,1\}$ that
\begin{equation}
r_{\varepsilon, i} =
\begin{cases}
	\nicefrac{\nu}{2} 																			&\colon \nu < 1 \\
	\min \{\nicefrac{1}{2}, \gamma \alpha + (-1)^{i} \varepsilon\}	&\colon \nu = 1  \\
	0																										&\colon \nu > 1,
\end{cases}
\end{equation}
let $F = ( F(\theta,x) )_{(\theta,x) \in \R^d \times \R^d} \colon \R^d \times \R^d \to \R$ and $f \colon \R^d \to \R$ be the functions which satisfy for all $\theta, x  \in \R^d$ that
\begin{equation}
F(\theta,x) = \tfrac{\alpha}{2} \norm{\theta-x}^2 \qandq f(\theta) = \EXP{F(\theta,X_1)},
\end{equation}
and let $\Theta \colon \N_0 \times \Omega \to \R^d$ be the stochastic process which satisfies for all $n \in \N$ that 
\begin{equation}
\begin{split}
\Theta_0 = \xi \qandq \Theta_n = \Theta_{n-1} - \tfrac{\gamma}{n^\nu} (\nabla_\theta F) (\Theta_{n-1},X_n).
\end{split}
\end{equation}
Then 
\begin{enumerate}[(i)]
\item \label{main_theorem:item1}
it holds for all $\theta \in \R^d$ that
$
f(\theta)
= 
\tfrac{\alpha}{2} \norm{\theta- \EXp{X_1}}^2  + 
\tfrac{\alpha}{2} \, \EXP{\norm{X_1 - \EXp{X_1}}^2},
$

\item \label{main_theorem:item2}
it holds that
$
\{\theta \in \R^d  \colon  f(\theta) = \inf\nolimits_{w \in \R^d} f(w)  \} = \{ \EXp{X_1} \},
$

\item \label{main_theorem:item3}
it holds for all $\theta \in \R^d$ that
$
\lll \theta - \EXp{X_1},(\nabla f)(\theta) \rrr   =  \alpha \norm{\theta - \EXp{X_1}}^2,
$

\item \label{main_theorem:item4}
it holds for all $\theta \in \R^d$ that
$
\norm{(\nabla f)(\theta)} =    \alpha \norm{\theta - \EXp{X_1}},
$

\item \label{main_theorem:item5}
it holds for all $\theta \in \R^d$ that
\begin{equation}
\EXP{\norm{(\nabla_\theta F) (\theta,X_1) - (\nabla f)(\theta)}^2} = \alpha^2 \, \EXP{ \norm{X_1 - \EXp{X_1}}^2},
\end{equation}

\item \label{main_theorem:item6}
for every $\varepsilon \in (0,\infty)$ there exist $C_0,C_1 \in (0,\infty)$ such that for all $n \in \N$ it holds that
\begin{equation}
C_0n^{-r_{\varepsilon,0}}
\leq
\big(\EXP{\norm{\Theta_n-\EXp{X_1}}^2}\big)^{\nicefrac{1}{2}} 
\leq
C_1  n^{-r_{\varepsilon,1}},
\end{equation}

and
\item \label{main_theorem:item7}
for every $\varepsilon \in (0,\infty)$ there exist $C_0,C_1 \in (0,\infty)$ such that for all $n \in \N$ it holds that
\begin{equation}
C_0  n^{-2r_{\varepsilon,0}} 
\leq
\Exp{f(\Theta_n)} - f(\EXp{X_1}) 
\leq
C_1 n^{-2r_{\varepsilon,1}}.
\end{equation}
\end{enumerate}
\end{theorem}

\begin{proof}[Proof of Theorem~\ref{main_theorem}]
First, note that items~\eqref{properties_loss:item1}, \eqref{properties_loss:item2}, \eqref{properties_loss:item5}, \eqref{properties_loss:item6}, and \eqref{properties_loss:item7} in Lemma~\ref{properties_loss} establish items~\eqref{main_theorem:item1}--\eqref{main_theorem:item5}.
In addition, observe that the hypothesis that $\P(X_1= \EXp{X_1}) < 1$ ensures that $\EXp{\norm{X_1-\EXp{X_1}}^2} > 0$.
Proposition~\ref{UB_nuleq1}, Corollary~\ref{UB_nueq1_3}, Lemma~\ref{UB_very_fast}, Proposition~\ref{LB}, Proposition~\ref{LB_nueq1}, and Lemma~\ref{LB_nubigger1} therefore prove item~\eqref{main_theorem:item6}.
Moreover, note that item~\eqref{main_theorem:item1} ensures that for all $n \in \N$ it holds that
\begin{equation}
\begin{split}
\Exp{f(\Theta_n)} - f(\EXp{X_1}) 
&= 
\Exp{\tfrac{\alpha}{2} \big(  \norm{\Theta_n- \EXp{X_1}}^2 + \EXP{\norm{X_1 - \EXp{X_1}}^2}\big)} \\
&\quad 
- \tfrac{\alpha}{2}\big( \norm{\EXp{X_1}- \EXp{X_1}}^2 + \EXP{\norm{X_1 - \EXp{X_1}}^2}  \big) \\
&=
\tfrac{\alpha}{2}\,\Exp{\norm{\Theta_n- \EXp{X_1}}^2}.
\end{split}
\end{equation} 
Combining this and item~\eqref{main_theorem:item6} establishes item~\eqref{main_theorem:item7}.
The proof of Theorem~\ref{main_theorem} is thus completed.
\end{proof}


\bibliographystyle{acm}
\bibliography{lower_bounds_bibfile}

\end{document}